\documentclass[12pt]{amsart}
\usepackage[active]{srcltx}
\usepackage{calc,amssymb,amsthm,amsmath,amscd, eucal,ulem, mathtools}
\usepackage{alltt}
\usepackage[left=1.2in,top=1.0in,right=1.2in,bottom=1.0in]{geometry}

\usepackage{marginnote}

\RequirePackage[dvipsnames,usenames]{color}

\input{kmacros3.sty}
\input{xy}
\xyoption{all}
\usepackage{tikz}

\numberwithin{equation}{theorem}

\renewcommand{\m}{\mathfrak{m}}

\DeclareMathOperator{\Ass}{Ass}

\DeclareMathOperator{\crk}{crk}
\DeclareMathOperator{\lc}{H}

\DeclareMathOperator{\cO}{\mathcal{O}}
\DeclareMathOperator{\fm}{\mathfrak{m}}

\usepackage{setspace}
\usepackage{hyperref}
\usepackage{enumerate}
\usepackage{graphicx}

\usepackage[all,cmtip]{xy}
\usepackage{verbatim}

\theoremstyle{theorem}

\begin{document}
\title{$D$-module and $F$-module length of local cohomology modules}
\author[Katzman]{Mordechai Katzman}
\address{Department of Pure Mathematics, University of Sheffield, Hicks Building, Sheffield S3 7RH, United Kingdom}
\email{M.Katzman@sheffield.ac.uk}
\author[Ma]{Linquan Ma}
\address{Department of Mathematics\\ University of Utah\\ Salt Lake City}
\email{lquanma@math.utah.edu}
\author[Smirnov]{Ilya Smirnov}
\address{Department of Mathematics\\ University of Michigan \\ Ann Arbor}
\email{ismirnov@umich.edu}
\author[Zhang]{Wenliang Zhang}
\address{Department of Mathematics\\ University of Illinois at Chicago \\ Chicago}
\email{wlzhang@uic.edu}


\thanks{W. Zhang is partially supported by the National Science Foundation through grant DMS \#1606414.}
\thanks{L. Ma is partially supported by the National Science Foundation through grant DMS \#1600198, and partially by the National Science Foundation CAREER Grant DMS \#1252860/1501102.}
\maketitle

\begin{abstract}
Let $R$ be a polynomial or power series ring over a field $k$. We study the length of local cohomology modules $\lc^j_I(R)$ in the category of $D$-modules and $F$-modules. We show that the $D$-module length of $\lc^j_I(R)$ is bounded by a polynomial in the degree of the generators of $I$. In characteristic $p>0$ we obtain upper and lower bounds on the $F$-module length in terms of the dimensions of Frobenius stable parts and the number of special primes of local cohomology modules of $R/I$. The obtained upper bound is sharp if $R/I$ is an isolated singularity, and the lower bound is sharp when $R/I$ is Gorenstein and $F$-pure. We also give an example of a local cohomology module that has different $D$-module and $F$-module lengths.
\end{abstract}

%Let $R$ be a polynomial or power series ring over a field $k$. We study the length of local cohomology modules $\lc^j_I(R)$ in the category of $D$-modules and, in characteristic $p$, the length of $\lc_I^c(R)$ with $c=\height I$ in the category of $F$-modules. We obtain the actual length of $\lc^j_I(R)$ as either a $D$-module or an $F$-module in several cases. We produce a general upper bound for $\lc^j_I(R)$ in the category of $D$-modules when $R$ is a polynomial ring. In characteristic $p>0$, if the dimension of the non-$F$-rational locus of $R/I$ is small, we also obtain upper bounds on the $D$-module length in terms of the dimensions of the Frobenius stable parts of certain local cohomology modules of $R/I$ and its localizations. Our bounds are sharp in many cases, for example, when $R/I$ has an isolated singularity. When $R/I$ is $F$-pure, we also obtain sharp lower bounds on the $F$-module length of $\lc_I^j(R)$ in terms of the number of special primes of $\lc_\m^{n-j}(R/I)$, and when $R/I$ is Cohen-Macaulay we can explicitly write down an $F$-module filtration of $\lc_I^c(R)$ that is maximal when $R/I$ is Gorenstein. Finally, we construct an example of a local cohomology module of $R$ such that, with its natural structure, its $D$-module length is strictly greater than its $F$-module length. We compute the Frobenius stable part of the top local cohomology module of the Fermat hypersurface $k[x_0,x_1,\dots,x_d]/(x_0^n+x_1^n+\cdots +x_d^n)$ explicitly in terms of the number of solutions to a system of equations on remainders.

\setcounter{tocdepth}{1}
%\marginnote{We added the table of contents for readability LM\&IS}
\tableofcontents

\section{Introduction}

Since its introduction by Grothendieck, local cohomology has become a major part of commutative algebra that has been studied from different points of view. When $R$ is a polynomial or power series ring over a field $k$, each local cohomology module $\lc^j_I(R)$ admits a natural module structure over $D(R,k)$, the ring of $k$-linear differential operators ($D(R,k)$-modules will be reviewed in \S\ref{section: preliminaries}). In characteristic 0, \cite{LyubeznikFinitenessLocalCohomology} shows that $\lc^j_I(R)$ has finite length as a $D(R,k)$-module, though $\lc^j_I(R)$ is rarely finitely generated as an $R$-module. To this day, using the finite length property of $\lc^j_I(R)$ in the category of $D(R,k)$-modules is still the only way to prove that $\lc^j_I(R)$ has finitely many associated primes in characteristic 0. In characteristic $p$, Frobenius action on local cohomology modules was used with great success by a number of authors, {\it e.g.}, \cite{PeskineSzpiroDimensionProjective}, \cite{HartshorneSpeiserLocalCohomologyInCharacteristicP} and \cite{HunekeSharpBassnumbersoflocalcohomologymodules}.
Lyubeznik (\cite{LyubeznikFModulesApplicationsToLocalCohomology})
conceptualizes the previous work to develop a theory of $F$-modules  in characteristic $p$
(the reader may find an overview in \S\ref{section: preliminaries}).

As we have seen, in characteristic $p$, local cohomology modules $\lc^j_I(R)$ can be viewed as both $D(R,k)$-modules and $F$-modules; \cite{LyubeznikFModulesApplicationsToLocalCohomology} compares these two points of view. It's shown that each $F$-module $M$ admits a natural $D(R,k)$-module structure and its length as an $F$-module, $l_{F_R}(M)$, is no more than its length as a $D(R,k)$-module, $l_{D(R,k)}(M)$. The comparison of these two points of view was continued in \cite{BlickleDmodulestructureofRFmodules}, where Blickle shows that over an algebraically closed field
the $D(R,k)$-module length is equal to the $F^{\infty}$-module length ($F^{\infty}$-modules will be reviewed in \S\ref{section: preliminaries}). The fact that local cohomology modules have finite length as $D(R,k)$-modules has found many applications; for instance \cite{NunezBetancourtWittGeneralizedLyubeznikNumbers} introduces numerical invariants of local rings using the length of local cohomology modules as $D(R,k)$-modules and shows that there are close connections between these invariants and $F$-singularities.

Despite the importance of the finiteness of the length of local cohomology modules as $D(R,k)$-modules and $F$-modules, finding the actual length of local cohomology modules as such modules remains an intriguing and difficult open question. In this paper we provide partial answers to this question in characteristic $p$.

\begin{theorem}[Theorems \ref{theorem--D-module length for isolated singularities} and \ref{theorem--lower bounds of F-module length for F-pure}]
Let $R=k[[x_1,\dots, x_n]]$ (or $k[x_1,\dots,x_n]$) with $\m=(x_1,\dots,x_n)$, where $k$ is a field of characteristic $p>0$. Let $A=R/I$ be reduced and equidimensional (respectively, graded reduced and equidimensional) of dimension $d\geq 1$.
\begin{enumerate}
\item If $A$ has an isolated non-$F$-rational point at $\fm$ ({\it e.g.}, $A$ has an isolated singularity at $\m$), then
\[l_{D(R,k)}(\lc^{n-d}_I(R))=\dim_k(0^*_{\lc^d_{\fm}(A)})_s+c=\dim_k((\lc_{\fm}^d(A))_s)+c\]
where $c$ is the number of minimal primes of $A$. Moreover, if $k$ is separably closed, then
\[l_{F_R}(\lc^{n-d}_I(R))=l_{F_R^e}(\lc^{n-d}_I(R))= l_{F_R^\infty}(\lc^{n-d}_I(R))= l_{D_R}(\lc^{n-d}_I(R))=\dim_k(0^*_{\lc^d_{\fm}(A)})_s+c.\]
\item If $A$ is $F$-pure and quasi-Gorenstein, then $l_{F_R}(\lc_I^{n-d}(R))$ is exactly the number of special primes of $\lc_\m^d(A)$.
\end{enumerate}
\end{theorem}

One ought to remark that there is an effective algorithm to compute special primes of $\lc_\m^d(A)$ (\cite{KatzmanZhangAlgorithmAnnihilatorAritinian}), hence the result above provides a practical tool to compute $l_{F_R}(\lc_I^{n-d}(R))$ when $A=R/I$ is $F$-pure and quasi-Gorenstein.

We also construct the first example of a local cohomology module over an algebraically closed field whose $D(R,k)$-module length disagrees with its $F$-module length.
\begin{theorem}[Proposition \ref{proposition--D-length bigger than F-length 2}]
\label{theorem: main example}
Let $R = \overline{\mathbb{F}}_{p}[x, y, z, t]$ with $p \equiv 4 \pmod 7$ and $f = tx^7 + ty^7 + z^7$. Then
\[
l_{F_R}(\lc^1_{f}(R))  = 3 < 7 = l_{F^\infty_R}(\lc^1_{f} (R))=l_{D(R,\overline{\mathbb{F}}_p)} (\lc^1_{f} (R)).
\]
\end{theorem}

While finding the actual length remains elusive in its full generality, we provide both upper and lower bounds.

\begin{theorem}[Theorem \ref{theorem: length of localization}]
\label{theorem: general bound on hypersurface}
Let $R=k[x_1,\dots,x_n]$ be a polynomial ring over a field $k$ and $f\in R$ be a polynomial of degree $d$. Then
\[l_{D(R,k)}(\lc^1_{f}(R))\leq (d+1)^n-1.\]
\end{theorem}

%\marginnote{Characteristic=p}
\begin{theorem}[Theorem \ref{theorem--Upper bounds for D-module length when singular locus has dimension 1}]
\label{theorem: formulas and bounds}
Let $R=k[x_1,\dots,x_n]$ and $\m=(x_1,\dots,x_n)$. Let $I$ be a homogeneous reduced and equidimensional ideal of $R$. Set $A=R/I$ with $\dim A=d\geq 2$. Suppose the non-$F$-rational locus of $A$ has dimension $\leq1$ (e.g., the nonsingular locus has dimension $\leq1$). Then we have
\begin{eqnarray*}
l_{D(R,k)}(\lc_I^{n-d}(R))&\leq &c+\sum_{\dim R/P=1}\dim_{\kappa(P)} (\lc_{P\widehat{A_P}}^{d-1}(\widehat{A_P}))_s+\dim_k({\lc^d_{\fm}(A)})_s\\
&=&c+\sum_{\dim R/P=1} \dim_{\kappa(P)}(0^*_{\lc_{P\widehat{A_P}}^{d-1}(\widehat{A_P})})_s+\dim_k(0^*_{\lc^d_{\fm}(A)})_s
\end{eqnarray*}
where $c$ is number of minimal primes of $I$.
\end{theorem}

%\marginnote{Characteristic=p}
\begin{theorem}[Theorem \ref{theorem--lower bounds of F-module length for F-pure}]
\label{theorem: F-length for F-pure}
Let $k$ be a field of positive characteristic, let $R$ denote the local ring $k[[x_1, \ldots, x_n]]$
(or the graded ring $k[x_1, \ldots, x_n]$), and $\m = (x_1, \ldots, x_n)$.
Let $A=R/I$ be reduced, equidimensional
(or, respectively, graded reduced and equidimensional), and $F$-pure of dimension $d\geq 1$.
Then  $l_{F_R}(\lc_I^{n-j}(R))$ is at least the number of special primes of $\lc_\m^j(A)$.
\end{theorem}

The upper and lower bounds on the $D(R,k)$-module and $F_R$-module length in the above results are sharp in many cases (see \S \ref{section: formulas and bounds}, \S \ref{section: F-pure}, \S \ref{section: diagonal hypersurfaces}),
 and we can explicitly describe an $F_R$-submodule filtration of $\lc_I^{n-d}(R)$ in terms of the generating morphisms when $R/I$ is Cohen-Macaulay, which is maximal when $R/I$ is Gorenstein and $F$-pure (see Theorem \ref{Theorem: F-length for CM rings} and \ref{Theorem: F-length for Gorenstein rings}).

We also construct an example (Example \ref{example: completion of simple no longer simple}) of a simple $D(R,k)$-module whose completion at a prime ideal $P$ is not a simple $D(\widehat{R}_P,\kappa(P))$-module.

Our paper is organized as follows. In Section \ref{section: preliminaries}, we recall some basic notions and results regarding $D(R,k)$-modules, $F$-modules, and tight closure theory. Section \ref{section: general bound} is concerned with Theorem \ref{theorem: general bound on hypersurface}. Section \ref{section: formulas and bounds} is devoted to proving Theorem \ref{theorem: formulas and bounds}. In Section \ref{section: F-pure} we prove Theorem \ref{theorem: F-length for F-pure}, we also describe explicitly the maximal $F_R$-module filtration of $\lc_I^{n-d}(R)$ in terms of their generating morphisms when $R/I$ is Gorenstein and $F$-injective (the Cohen-Macaulay $F$-injective case will also be discussed).  In Section \ref{section: diagonal hypersurfaces}, we compute the dimension of the stable part (under the natural Frobenius action) of the top local cohomology of Fermat hypersurfaces. \S \ref{section: examples} proves Theorem \ref{theorem: main example} and related results. Examples and remarks showing the sharpness of our bounds will be given throughout.

\subsection*{Acknowledgements}
This material is based partly upon work supported by the National Science Foundation under Grant No.1321794;
some of this work was done at the Mathematics Research Community (MRC) in Commutative Algebra in June 2015.
The second, third and fourth authors would like to thank the staff and organizers of the MRC and the American Mathematical Society for the support provided.
The first author thanks the Department of Mathematics of the University of Illinois at Chicago for its hospitality while working on this and other projects. The authors thank the referee for valuable comments.

\section{Preliminaries}
\label{section: preliminaries}
Throughout this paper, we always assume that $R=k[[x_1,\dots, x_n]]$ or $R=k[x_1,\dots,x_n]$ where $k$ is a field (not necessarily algebraically closed or perfect), and $A$ is a reduced and equidimensional quotient or graded reduced and equidimensional quotient of $R$ of dimension $d\geq 1$. We set $\fm=(x_1,\dots,x_n)$. In this section we collect some notations and preliminary results from \cite{LyubeznikFModulesApplicationsToLocalCohomology} and \cite{BlickleIntersectionhomologyDmodule}.

\subsection*{$D$-modules} The differential operators $\delta$: $R\to R$ of order $\leq n$ can be defined inductively as follows. A differential operator of order $0$ is just multiplication by an element of $R$. A differential operator of order $\leq n$ is an additive map $\delta$: $R\to R$ such that for every $r\in R$, the commutator $[\delta, r]=\delta\circ r-r\circ\delta$ is a differential operator of order $\leq n-1$. The differential operators form a ring with the multiplication defined via the composition. We denote this ring by $D_R$.

We denote by $D(R,k)\subseteq D_R$ the subring of $D_R$ consisting of all $k$-linear differential operators. Since $R=k[[x_1,\dots, x_n]]$ or $R=k[x_1,\dots,x_n]$, it can be verified that $D(R, k)$ is generated by all operators of the form $\frac{1}{j!}\frac{\partial^{j}}{\partial x_i^{j}}$. By a {\it $D_R$-module} or a {\it $D(R,k)$-module} we mean a left module over $D_R$ or $D(R,k)$.

When $k$ is a field of characteristic $p$, it is not hard to show that every differential operator of order $\leq p^e-1$ is $R^{p^e}$-linear, where $R^{p^e}\subseteq R$ is the subring of all the $p^e$-th powers of all the elements of $R$. In other words, we always have $D_R$ is a subring of $\cup_e\Hom_{R^{p^e}}(R,R)$.\footnote{In fact $D_R=\cup_{e}\Hom_{R^{p^e}}(R,R)$ when $R$ is $F$-finite, i.e., $R$ is finitely generated as a $R^{p}$-module \cite{YekutieliExplicitconstructionofResidueComplex}.} In particular, all differential operators are automatically $k$-linear if $k$ is perfect, thus $D_R=D(R,k)$ if $k$ is perfect.

\subsection*{$F$-modules} Assume that $k$ is a field of characteristic $p$. The notion of $F$-modules was introduced in \cite{LyubeznikFModulesApplicationsToLocalCohomology}, and further investigated and generalized in \cite{BlickleThesis, BlickleDmodulestructureofRFmodules}. We use $R^{(e)}$ to denote the target ring of the $e$-th Frobenius map $F^e$: $R\rightarrow R$. We shall let $F^e(-)$ denote the Peskine-Szpiro's Frobenius functor from $R$-modules to $R$-modules. In detail, $F^e(M)$ is given by base change to $R^{(e)}$ and then identifying $R^{(e)}$ with $R$, i.e., $F^e(M)=R^{(e)}\otimes_R M$.

An {\it $F^e_R$-module} is an $R$-module $M$ equipped with an $R$-linear isomorphism $\theta$: $M\rightarrow F^e(M)$ which we call the structure morphism of $M$. A homomorphism of $F^e_R$-modules is an $R$-module homomorphism $f$: $M\rightarrow M'$ such that the following diagram commutes
\[\xymatrix{
M\ar[r]^{f}\ar[d]^{\theta} & M'\ar[d]^{\theta'}\\
F^e(M)\ar[r]^{F^e(f)}&F^e(M')\\
}\]
When $e=1$ we simply say $M$ is an $F_R$-module (or $F$-module if $R$ is clear from the context). It is easy to see that every $F^e_R$-module is also an $F^{er}_R$-module for every $r\geq 1$ by iterating the structure isomorphism $r$ times. The union of the categories of $F^e_R$-modules over all $e$ forms what we called the category of $F^\infty_R$-modules.\footnote{In \cite{BlickleThesis, BlickleDmodulestructureofRFmodules, BlickleIntersectionhomologyDmodule},
$F^e_R$-modules and $F^\infty_R$-modules are called unit $R[F^e]$-modules and unit $R[F]$-modules respectively. In this paper we will use Lyubeznik's notation \cite[Remark 5.6]{LyubeznikFModulesApplicationsToLocalCohomology} since we think this is more natural in comparison with the usual $F_R$-modules.} With these definitions, the categories of $F^e_R$-modules and $F^\infty_R$-modules are abelian categories.

When $R=k[x_1,\dots,x_n]$ and $M$ is a graded $R$-module, there is a natural grading on $F(M)=R^{(e)}\otimes_R M$ given by $\deg(r\otimes m)=\deg r+p^e\cdot\deg m$ for homogeneous elements $r\in R$ and $m\in M$. With this grading, a {\it graded $F^e_R$-module} is an $F^e_R$-module $M$ such that the structure isomorphism $\theta$ is degree-preserving. A morphism of graded $F^e_R$-modules is a degree-preserving morphism of $F^e_R$-modules. It is not hard to see that graded $F^e_R$-modules form an abelian subcategory of the category of $F^e_R$-modules. Graded $F^e_R$-modules, at least for $e=1$, were introduced and studied in detail in \cite{ZhangYGradedFModules} and \cite{MaZhangEuleriangradedDmodules}.

A generating morphism of an $F^e_R$-module $M$ is an $R$-module homomorphism $\beta$: $M_0\rightarrow F^e(M_0)$, where $M_0$ is some $R$-module, such that $M$ is the limit of the inductive system in the top row of the commutative diagram
\[\xymatrix{
M_0\ar[r]^{\beta}\ar[d]^{\beta} & F^e(M_0)\ar[d]^{F^e(\beta)}\ar[r]^{F^e(\beta)}& F^{2e}(M_0)\ar[r]^{\quad F^{2e}(\beta)}\ar[d]^{F^{2e}(\beta)}&{\cdots}\\
F^e(M_0)\ar[r]^{F^e(\beta)}& F^{2e}(M_0)\ar[r]^{F^{2e}(\beta)}&F^{3e}(M_0)\ar[r]^{\quad F^{3e}(\beta)}&{\cdots}\\
}\]
and $\theta$: $M\rightarrow F^e(M)$, the structure isomorphism of $M$, is induced by the vertical arrows in this diagram. $M_0$ is called a {\it root} of $M$ if $\beta$: $M_0\to F^e(M_0)$ is injective. An $F^e_R$-module $M$ is called {\it $F$-finite} if $M$ has a generating morphism $\beta$: $M_0\rightarrow F^e(M_0)$ with $M_0$ a finitely generated $R$-module. When $M_0$ is graded and $\beta$ is degree-preserving, we say that $M$ is {\it graded $F$-finite $F^e_R$-module}.

It is a fundamental result of Lyubeznik (\cite{LyubeznikFModulesApplicationsToLocalCohomology}) that local cohomology modules $\lc_I^i(R)$ have a natural structure of $F$-finite $F^e_R$-modules for every $e\geq 1$.
Moreover, when $R=k[x_1,\dots,x_n]$ and $I$ is a homogeneous ideal of $R$, $\lc_I^i(R)$ are graded $F$-finite $F^e_R$-modules \cite{ZhangYGradedFModules}.

Following \cite{LyubeznikFModulesApplicationsToLocalCohomology}, for any $F$-finite $F_R$-module $M$, there exists a smallest $F_R$-submodule $N\subseteq M$ with the property that $M/N$ is supported only at $\m$. Hence $M/N$ is isomorphic (as an $R$-module) to $E^{\oplus r}$ where $E=E_R(k)$ denotes the injective hull of $k$. We define $\crk(M)$, the {\it the corank of $M$,} to be $r$. %A similar construction works in the graded case, and in fact we have a stronger result in the graded case \cite{LyubeznikSinghWaltherlocalcohomologysupportedatdeterminantalideals}

One important feature of $F_R$-modules is that they have a natural structure of $D_R$-modules, and thus $D(R,k)$-modules. We briefly recall this here, and we refer to \cite[Section 5]{LyubeznikFModulesApplicationsToLocalCohomology} for details. Let $M$ be an $F_R$-module with structure isomorphism $\theta$. We set $\theta_e$ to be the $e$-th iterate of $\theta$, i.e., $$\theta_e=F^{e-1}(\theta)\circ\cdots\circ F(\theta)\circ\theta: M\to F^e(M). $$  Now every element $\delta\in\Hom_{R^{p^e}}(R,R)$ acts on $F^e(M)=R^{(e)}\otimes_RM$ via $\delta\otimes\id_M$. We let $\delta$ act on $M$ via $\theta_e^{-1}\circ(\delta\otimes\id_M)\circ\theta_e$. It is not very difficult to check that this action is well-defined. Moreover, an entirely similar construction shows that every $F^e_R$-module also have a canonical structure of a $D_R$-module. It follows that $F_R^\infty$-modules are naturally $D_R$-modules, and thus $D(R,k)$-modules. In sum, we have the following inclusion of abelian categories:

%\marginnote{$F_R^\infty$ needs definition: We defined $F_R^\infty$ on Page 5, the paragraph after the diagram. Do we need to say more words on it?--Linquan}

\[
\{F_R-\mbox{modules} \}\subseteq \{F_R^e-\mbox{modules}\} \subseteq \{F_R^\infty-\mbox{modules}\}\subseteq \{D_R-\mbox{modules}\} \subseteq \{D(R,k)-\mbox{modules}\}.
\]
Therefore for any $F_R$-module $M$ we have the following inequalities on its length considered in the corresponding categories:
\begin{equation}
\label{equation--basic relation on length in different categories}
l_{F_R}(M)\leq l_{F_R^e}(M)\leq l_{F_R^\infty}(M)\leq l_{D_R}(M)\leq l_{D(R,k)}(M).
\end{equation}

\subsection*{$A\{f\}$-modules}
Recall that we always assume $A$ is a reduced and equidimensional or graded reduced and equidimensional quotient of $R$ of dimension $d\geq 1$. Assume also that $k$ is a field of characteristic $p$. Let $M$ be a (in most cases Artinian) module over $A$. We say that $M$ is an $A\{f\}$-module if there is an additive map $f\colon M\to M$ such that $f(am)=a^pf(m)$. We will use $M_s = \bigcap k\langle f^i(M)\rangle$ to denote the Frobenius stable part of $M$, note that this is a $k$-vector space, and depends not only on $M$ but also on the action $f$ on $M$. It is well-known that when $M$ is either Noetherian or Artinian over $A$, $M_s$ is a finite dimensional $k$-vector space \cite{HartshorneSpeiserLocalCohomologyInCharacteristicP}, \cite{LyubeznikFModulesApplicationsToLocalCohomology}.

Let $\scr{H}_{R,A}$ denote the Lyubeznik functor introduced in \cite{LyubeznikFModulesApplicationsToLocalCohomology}: for every Artinian $A\{f\}$-module $M$, we have a natural induced map $\alpha$: $F_R(M)\to M$, since $F_R(M)^\vee\cong F_R(M^\vee)$ by \cite[Lemma 4.1]{LyubeznikFModulesApplicationsToLocalCohomology}, we define $$\scr{H}_{R,A}(M)=\varinjlim(M^\vee\xrightarrow{\alpha^\vee} F_R(M^\vee)\xrightarrow{F(\alpha^\vee)} F_R^2(M^\vee)\to\cdots),$$ which is an $F$-finite $F_R$-module. In the graded case we have a similar functor ${}^*\scr{H}_{R,A}$ that takes a graded Artinian $A\{f\}$-module to a graded $F$-finite $F_R$-module: one needs to replace Matlis dual by graded Matlis dual in the construction of ${}^*\scr{H}_{R,A}$, (see \cite{LyubeznikSinghWaltherlocalcohomologysupportedatdeterminantalideals} for details). One important example that we will use repeatedly is that $\scr{H}_{R, A}(\lc_{\fm}^d(A))\cong \lc_I^{n-d}(R)$, and in the graded case we also have ${^*}\scr{H}_{R, A}(\lc_{\fm}^d(A))\cong \lc_I^{n-d}(R)$ \cite[Example 4.8]{LyubeznikFModulesApplicationsToLocalCohomology}, \cite[Proposition 2.8]{LyubeznikSinghWaltherlocalcohomologysupportedatdeterminantalideals}.

\subsection*{Tight closure and $F$-singularities} Tight closure theory was introduced by Hochster-Huneke in \cite{HochsterHunekeTC1}. In this article, we only need some basic properties of tight closure of zero in the top local cohomology module, $0^*_{\lc_\m^d(A)}$. Under mild conditions, for example when $(A,\m)$ is an excellent local domain of dimension $d$, $0^*_{\lc_\m^d(A)}$ is the largest proper $A$-submodule of $\lc_\m^d(A)$ that is stable under the natural Frobenius action on $\lc_\m^d(A)$ (cf.~\cite{SmithFRatImpliesRat}). A local ring $(A,\m)$ is called {\it $F$-rational} if $A$ is Cohen-Macaulay and $0^*_{\lc_\m^d(A)}=0$.\footnote{This is not the original definition of $F$-rationality, but is shown to be equivalent (\cite{SmithFRatImpliesRat}).} Under mild conditions on the ring, for example when $(A,\m)$ is an excellent, reduced and equidimensional local ring, $A$ is $F$-rational on the punctured spectrum if and only if $0^*_{\lc_\m^d(A)}$ has finite length.

A local ring $(A,\m)$ of characteristic $p>0$ is called {\it $F$-pure} if the Frobenius endomorphism $F$: $A\rightarrow A$ is pure.\footnote{A map of $A$-modules $N\rightarrow N'$ is pure if for every $A$-module $M$ the map $N\otimes_AM\rightarrow N'\otimes_AM$ is injective. This implies that $N\rightarrow N'$ is injective, and is weaker than the condition that $0\rightarrow N\rightarrow N'$ be split.} Under mild conditions, for example when the Frobenius map $A\xrightarrow{F}A$ is a finite map or when $A$ is complete, $F$-purity of $A$ is equivalent to the condition that the Frobenius endomorphism $A\xrightarrow{F} A$ is split \cite[Corollary 5.3]{HochsterRobertsFrobeniusLocalCohomology}. The Frobenius endomorphism on $A$ induces a natural Frobenius action on each local cohomology module $\lc_\m^i(A)$ and we say $(A,\m)$ is {\it $F$-injective} if this natural Frobenius action on $\lc_\m^i(A)$ is injective for every $i$ (\cite{FedderFPureRat}). This holds if $A$ is $F$-pure \cite[Lemma 2.2]{HochsterRobertsFrobeniusLocalCohomology}. For some other basic properties of $F$-pure and $F$-injective singularities, see \cite{HochsterRobertsFrobeniusLocalCohomology}, \cite{FedderFPureRat}, \cite{EnescuHochsterTheFrobeniusStructureOfLocalCohomology}.

\section{A general bound of local cohomology modules as $D$-modules}
\label{section: general bound}
In this section, we will establish a bound of $\lc^j_I(R)$ as a $D(R,k)$-module when $R=k[x_1,\dots,x_n]$ is a polynomial ring over a field $k$ (of any characteristic) in terms of the degrees of generators of $I$. To this end, we begin with recalling the notion of the Bernstein filtration and $k$-filtration.

Denote $\frac{1}{t!}\frac{\partial^t}{\partial x^t_i}$ by $D_{t,i}$ for $t\in\mathbb{N},1\leq i\leq n$. Then $D(R,k)=R\langle D_{t,i}\mid t\in \mathbb{N}, 1\leq i\leq n\rangle$. We set $\scr{F}_j$ to be the $k$-linear span of the set of products
\[\{x^{i_1}_1\cdots x^{i_n}_n\cdot D_{t_1,1}\cdots D_{t_n,n}|i_1+\cdots+i_n+t_1+\cdots+t_n\leq j\}.\]
Then the {\it Bernstein} filtration on $D(R,k)$ (\cite[Definition 2.6]{LyubeznikCharFreeHolonomic}) is defined to be $k=\scr{F}_0\subset \scr{F}_1\subset \scr{F}_2 \cdots$

\begin{definition}[Definition 3.2 in \cite{LyubeznikCharFreeHolonomic}]
A {\it $k$-filtration} on a $D(R,k)$-module $M$ is an ascending chain of finite-dimensional $k$-vector spaces $M_0\subset M_1\subset \cdots$ such that $\cup_i M_i=M$ and $\scr{F}_i M_j\subset M_{i+j}$ for all $i$ and $j$.
\end{definition}

We need the following result of Lyubeznik. We note that in characteristic $0$, the $D$-module length of holonomic $D$-modules has been studied before, see \cite{BernsteinModulesoverRingofDifferential}, \cite{BjorkRingsofDifferentialOperators}.

\begin{theorem}[Theorem 3.5 in \cite{LyubeznikCharFreeHolonomic}]
\label{theorem: bound by multiplicity}
Let $M$ be a $D(R,k)$-module with a $k$-filtration $M_0\subset M_1\subset \cdots$. Assume there is a constant $C$ such that $\dim_k(M_i)\leq Ci^n$ for sufficiently large $i$. Then the length of $M$ as a $D(R,k)$-module is at most $n!C$.
\end{theorem}

The statement of \cite[Theorem 3.5]{LyubeznikCharFreeHolonomic} assumes that there is a constant $C$ such that $\dim_k(M_i)\leq Ci^n$ for {\it all} $i\geq 0$; however, the proof of \cite[Theorem 3.5]{LyubeznikCharFreeHolonomic} only uses the fact that there is a constant $C$ such that $\dim_k(\scr{M}_i)\leq Ci^n$ for {\it sufficiently large} $i$. Hence the proof of Theorem \ref{theorem: bound by multiplicity} is identical to the one of \cite[Theorem 3.5]{LyubeznikCharFreeHolonomic}, and is omitted.

To illustrate the advantage of only requiring $\dim_k(M_i)\leq Ci^n$ for {\it sufficiently large} $i$, we consider a simple example.

\begin{example}
\label{example: length of R}
Set $R_i$ to be the $k$-span of monomials in $x_1,\dots,x_n$ of degree at most $i$. It is clear that $R_0\subset R_1\subset \cdots$ is a $k$-filtration of $R$. It is well-known that $\dim_k(R_i)=\binom{n+i}{i}$, which is a polynomial in $i$ of degree $n$ with leading coefficient $\frac{1}{n!}$. Hence, given any $\varepsilon>0$, we have $\dim_k(R_i)\leq \frac{1+\varepsilon}{n!}i^n$ for sufficiently large\footnote{We should remark that how large $i$ needs to be certainly depends on $\varepsilon$.} $i$. Since the length of a module is an integer, it follows from Theorem \ref{theorem: bound by multiplicity} that the length of $R$ in the category of $D(R,k)$-modules is 1. On the other hand, if one requires $\dim_k(R_i)\leq Ci^n$ for {\it all} $i$, then one will need $C\geq \frac{n+1}{n!}$ (consider the case when $i=1$) and consequently can not deduce the correct length of $R$ from \cite[Theorem 3.5]{LyubeznikCharFreeHolonomic}.
\end{example}

\begin{remark}
\label{remark: induced bound on localization}
In the proof of \cite[Corollary 3.6]{LyubeznikCharFreeHolonomic}, the following statement is proved: {\it Let $M$ be a $D(R,k)$-module with a $k$-filtration $M_0\subseteq M_1\subseteq \cdots$ and $f\in R$ be a polynomial of degree $d$. If there is a constant $C$ such that $\dim_k(M_i)\leq Ci^n$ for sufficiently large $i$, then $M'_i=\{\frac{m}{f^i}\mid m\in M_{i(d+1)}\}$ induces a $k$-filtration of $M_f$ such that $\dim_k(M'_i)\leq C(d+1)^ni^n$ for sufficiently large $i$.}
\end{remark}

\begin{theorem}
\label{theorem: length of localization}
Let $f\in R$ be a polynomial of degree $d$. Then the length of $R_f$ in the category of $D(R,k)$-modules is at most $(d+1)^n$. And, the length of $\lc^1_{f}(R)$ is at most $(d+1)^n-1$.
\end{theorem}
\begin{proof}
Combining Example \ref{example: length of R} and Remark \ref{remark: induced bound on localization}, we see that $R_f$ admits a $k$-filtration $R'_0\subset R'_1\subset\cdots$ such that, for any any $\varepsilon>0$, one has $\dim_k(R'_i)\leq \frac{1+\varepsilon}{n!}(d+1)^ni^n$. According to Theorem \ref{theorem: bound by multiplicity}, we have that the length of $R_f$ in the category of $D(R,k)$-modules is at most $(1+\varepsilon)(d+1)^n$ for any $\varepsilon>0$; it follows that the length of $R_f$ in the category of $D(R,k)$-modules is at most $(d+1)^n$.

The second conclusion follows from the exact sequence $0\to R\to R_f\to \lc^1_{f}(R)$ and the fact that the length of $R$ is 1.
\end{proof}

\begin{corollary}
\label{corollary: length of local cohomology in terms of degree}
Let $I$ be an ideal of $R$. If $I$ is generated by $f_1,\dots,f_t$ with $\deg(f_i)=d_i$, then the length of $\lc^j_I(R)$ in the category of $D(R,k)$-modules is at most
\[\sum_{1\leq i_1\leq \cdots\leq i_j\leq t}(d_{i_1}+\cdots+d_{i_j}+1)^n-1.\]
\end{corollary}
\begin{proof}
It is clear from Theorem \ref{theorem: length of localization} that the length of $\bigoplus R_{f_{i_1}\cdots f_{i_j}}$ is at most $$\sum_{1\leq i_1\leq \cdots\leq i_j\leq t}(d_{i_1}+\cdots+d_{i_j}+1)^n.$$ Our corollary follows from the fact that $\lc^j_I(R)$ is a proper subquotient of $\bigoplus R_{f_{i_1}\cdots f_{i_j}}$ in the category of $D(R,k)$-modules.
\end{proof}

The bounds in Theorem \ref{theorem: length of localization} and Corollary \ref{corollary: length of local cohomology in terms of degree}, though general, are very coarse. In the rest of the paper, we will focus on the length of $\lc^c_I(R)$ where $c$ is the height of $I$ and $k$ is of prime characteristic $p$, where $F$-module theory and tight closure theory can be used to produce sharper bounds.

%%%%%%%%%%%%%%%%%%%%%%%%%%%%%%%%%%%%%%%%%%%%%%%%%%%%%%%%%%%%%%%%%%%%%%%%%%%%%%%%%%%%%%%%%%%%%%%%%%%%%%%%%%%%%%%%%%%%%%%
\section{Formulas and upper bounds on the $D$-module and $F$-module length}
\label{section: formulas and bounds}
%%%%%%%%%%%%%%%%%%%%%%%%%%%%%%%%%%%%%%%%%%%%%%%%%%%%%%%%%%%%%%%%%%%%%%%%%%%%%%%%%%%%%%%%%%%%%%%%%%%%%%%%%%%%%%%%%%%%%%%
\textbf{Notation}: Henceforth  $R$ denotes $k[[x_1,\dots, x_n]]$ or $k[x_1,\dots,x_n]$ where $k$ is a field of characteristic $p>0$,
$\m=(x_1,\dots,x_n)$,  and we let $A=R/I$ be reduced and equidimensional or graded reduced and equidimensional of dimension $d\geq 1$.

We first analyze the case when $A$ has an isolated non-$F$-rational point at $\{\m\}$. We start with a few lemmas.

\begin{lemma}
\label{lemma--left exactness of taking stable part}
Let $0\rightarrow L\xrightarrow{\alpha} M\xrightarrow{\beta} N\rightarrow 0$ be an exact sequence of Artinian $A\{f\}$-modules. Let $f_L, f_M, f_N$ denote the
Frobenius actions on $L,M,N$ respectively. Then the stable parts form a left exact sequence of finite dimensional vector spaces: $0\rightarrow
L_s\xrightarrow{\alpha} M_s\xrightarrow{\beta} N_s$.
\end{lemma}
\begin{proof}
Let $\mathbb{K}$ be the perfect closure of $k$; define $A^{\mathbb{K}}=A\otimes_k \mathbb{K}$
and for any $A\{f\}$-module $X$, let $X^{\mathbb{K}}$ be the $A^{\mathbb{K}}\{ f \}$-module
$A^{\mathbb{K}} \otimes_k X$ where the action of $f$ is given by
$f\left( \sum_{i=1}^\mu x_i \otimes \lambda_i \right)= \sum_{i=1}^\mu f(x_i) \otimes \lambda_i^p$
for $x_1, \dots, x_\mu\in X$ and $\lambda_1, \dots, \lambda_\mu\in \mathbb{K}$.
\cite[Proposition 4.9]{LyubeznikFModulesApplicationsToLocalCohomology} implies that $X^{\mathbb{K}}_s=X_s \otimes_k \mathbb{K}$, so if we could show that
$0\rightarrow L_s^{\mathbb{K}} \xrightarrow{\alpha^{\mathbb{K}}} M_s^{\mathbb{K}}\xrightarrow{\beta^{\mathbb{K}}} N_s^{\mathbb{K}}$, the fact that
${\mathbb{K}}$ is faithfully flat over $k$ would imply the exactness of
$0\rightarrow L_s\xrightarrow{\alpha} M_s\xrightarrow{\beta} N_s$.
Therefore we may assume that $k$ is a perfect field.

The exactness at $L_s$ and $\beta\alpha=0$ are obvious. Hence to prove exactness it suffices to show that $\ker\beta \subseteq \im \alpha$.  Since $N$ is Artinian, by \cite[Proposition 4.4]{LyubeznikFModulesApplicationsToLocalCohomology} $\cup_{r\geq1}\ker f_N^r=\ker f_N^{r_0}$ for some $r_0$ sufficiently large and $\ker f_N\subseteq \ker f_N^2\subseteq \cdots$ stabilizes.

Pick $x \in M_s$ such that $\beta(x)=0$. Since we are assuming that $k$ is perfect, the $k$-vector space generated by $f^r(M)$ is just $f^r(M)$ for all $r\geq 0$;
therefore, for each $r\geq r_0$, there exists $y \in M$  such that $x=f_M^r(y)$. We have
$ f_N^r(\beta(y))=\beta( f_M^r(y))=0$. So $f_N^{r_0}(\beta(y))=0$ and hence $z=f_M^{r_0}(y)\in L$. Therefore for each $r\geq r_0$, there exists $z\in L$ such that
$x= f_M^{r-r_0}(z)$, hence $x\in L_s$.
\end{proof}

\begin{lemma}
\label{lemma--stable parts of tight closure of top local cohomology modules}
Suppose $\dim A=d \geq 1$, then $(0^*_{\lc_{\fm}^d(A)})_s\cong (\lc_{\fm}^d(A))_s$.
\end{lemma}
\begin{proof}
We have
$0\to 0^*_{\lc_{\fm}^d(A)}\to \lc_{\fm}^d(A)\to \lc_{\fm}^d(A)/0^*_{\lc_{\fm}^d(A)}\to 0$.
By Lemma \ref{lemma--left exactness of taking stable part}, it suffices to show that $(\lc_{\fm}^d(A)/0^*_{\lc_{\fm}^d(A)})_s=0$. But by \cite[Proposition 4.10]{LyubeznikFModulesApplicationsToLocalCohomology}, $\dim(\lc_{\fm}^d(A)/0^*_{\lc_{\fm}^d(A)})_s= \crk\scr{H}_{R,A}(\lc_{\fm}^d(A)/0^*_{\lc_{\fm}^d(A)})$.
Let $P_1,\dots, P_c$ be all the minimal primes of $A$. By \cite[Theorem 4.3 and 4.4]{BlickleIntersectionhomologyDmodule}, $\scr{H}_{R,A}(\lc_{\fm}^d(A)/0^*_{\lc_{\fm}^d(A)})$ is a direct sum of simple $F_R$-modules, each has $P_i$ as its unique associated prime. This implies $\crk\scr{H}_{R,A}(\lc_{\fm}^d(A)/0^*_{\lc_{\fm}^d(A)})=0$ because $\dim A\geq 1$, and hence $(\lc_{\fm}^d(A)/0^*_{\lc_{\fm}^d(A)})_s=0$ as desired.
\end{proof}

\begin{theorem}
\label{theorem--D-module length for isolated singularities}
Let $R=k[[x_1,\dots, x_n]]$ (or $k[x_1,\dots,x_n]$) with $\m=(x_1,\dots,x_n)$, where $k$ is a field of characteristic $p>0$. Let $A=R/I$ be reduced and equidimensional (respectively, graded reduced and equidimensional) of dimension $d\geq 1$. Assume that $A$ has an isolated non-$F$-rational point at $\fm$ (e.g., $A$ has an isolated singularity at $\m$). Then $$l_{D(R,k)}(\lc^{n-d}_I(R))=\dim_k(0^*_{\lc^d_{\fm}(A)})_s+c=\dim_k((\lc_{\fm}^d(A))_s)+c$$ where $c$ is the number of minimal primes of $A$. Moreover, if $k$ is separably closed, then we also have $$l_{F_R}(\lc^{n-d}_I(R))=l_{F_R^e}(\lc^{n-d}_I(R))= l_{F_R^\infty}(\lc^{n-d}_I(R))= l_{D_R}(\lc^{n-d}_I(R))=\dim_k(0^*_{\lc^d_{\fm}(A)})_s+c.$$
\end{theorem}

\begin{proof} The second equality follows immediately from Lemma \ref{lemma--stable parts of tight closure of top local cohomology modules}. Thus it suffices to show $l_{D(R,k)}(\lc^{n-d}_I(R))=\dim_k(0^*_{\lc^d_{\fm}(A)})_s+c$ and $l_{F_R}(\lc^{n-d}_I(R))=\dim_k(0^*_{\lc^d_{\fm}(A)})_s+c$ when $k$ is separably closed, since the other equalities would follow from (\ref{equation--basic relation on length in different categories}).

The short exact sequence $0\to 0^*_{\lc_{\fm}^d(A)}\to \lc_{\fm}^d(A)\to \lc_{\fm}^d(A)/0^*_{\lc_{\fm}^d(A)}\to 0$ induces:
\[0\to \scr{H}_{R, A}(\lc_{\fm}^d(A)/0^*_{\lc_{\fm}^d(A)})\to \scr{H}_{R, A}(\lc_{\fm}^d(A))\cong \lc_I^{n-d}(R)\to \scr{H}_{R, A}(0^*_{\lc_{\fm}^d(A)})\to 0.\]
Now by \cite[Corollary 4.2 and Theorem 4.4]{BlickleIntersectionhomologyDmodule}, $\scr{H}_{R, A}(\lc_{\fm}^d(A)/0^*_{\lc_{\fm}^d(A)})$ is a direct sum of simple $D(R,k)$-modules, each supported at a different minimal prime of $A$, so its $D(R,k)$-module length is $c$, and in fact each of these simple $D(R,k)$-modules are (simple) $F_R$-modules \cite[Theorem 4.3]{BlickleIntersectionhomologyDmodule}. Thus $l_{D(R,k)}(\scr{H}_{R, A}(\lc_{\fm}^d(A)/0^*_{\lc_{\fm}^d(A)}))=l_{F_R}(\scr{H}_{R, A}(\lc_{\fm}^d(A)/0^*_{\lc_{\fm}^d(A)}))=c$.

Since $A$ has an isolated non-$F$-rational point at $\m$, $0^*_{\lc^d_{\fm}(A)}$ has finite length as an $A$-module. This implies that $\scr{H}_{R,A}(0^*_{\lc^d_{\fm}(A)})$ is supported only at $\fm$. Hence $\scr{H}_{R,A}(0^*_{\lc^d_{\fm}(A)})$, {\it as a $D(R,k)$-module}, is a direct sum of finitely many copies of $\lc^n_{\fm}(R)$ by
\cite[Lemma (c)]{LyubeznikInjectivedimensionofDmodulescharacteristicfreeapproach}.
Moreover, {\it when $k$ is separably closed}, $\scr{H}_{R,A}(0^*_{\lc^d_{\fm}(A)})$ is a direct sum of copies of $E=\lc^n_{\fm}(R)$ even {\it as an $F_R$-module} \cite[Lemma 4.3]{MaThecategoryofF-moduleshasfiniteglobaldimension}.  The number of copies of $\lc^n_{\fm}(R)$ is exactly its $D(R,k)$-module length ($F_R$-module length when $k$ is separably closed). However, since $\scr{H}_{R,A}(0^*_{\lc^d_{\fm}(A)})$ is an $F_R$-module supported only at $\m$,  the number of copies of $\lc^n_{\fm}(R)$ is by definition $\crk\scr{H}_{R,A}(0^*_{\lc^d_{\fm}(A)})$, which is $\dim_k(0^*_{\lc^d_{\fm}(A)})_s$ by \cite[Proposition 4.10]{LyubeznikFModulesApplicationsToLocalCohomology}. Hence $l_{D(R,k)}(\scr{H}_{R,A}(0^*_{\lc^d_{\fm}(A)}))=\dim_k(0^*_{\lc^d_{\fm}(A)})_s$, and $l_{F_R}(\scr{H}_{R,A}(0^*_{\lc^d_{\fm}(A)}))=\dim_k(0^*_{\lc^d_{\fm}(A)})_s$ when $k$ is separably closed.
\end{proof}

\begin{remark}
\label{remark--F-module length for isolated singularities}
When $k$ is not separably closed, $l_{F_R}(\lc^{n-d}_I(R))=\dim_k(0^*_{\lc^d_{\fm}(A)})_s+c$ fails to hold in general, see Corollary \ref{corollary--D-length bigger than F-length 1}. However, if $k$ is a {\it finite field}, then we always have $l_{F_R^\infty}(\lc^{n-d}_I(R))=\dim_k(0^*_{\lc^d_{\fm}(A)})_s+c$, see Proposition \ref{proposition--F-module length for isolated sing over finite field}.
\end{remark}

\begin{remark}
\label{remark--D-module length for graded isolated singularities}
When $I$ is a homogeneous reduced and equidimensional ideal in $R=k[x_1,\dots,x_n]$ (i.e., $A$ is a graded domain), it is easy to check that $(0^*_{\lc^d_{\fm}(A)})_s=({\lc^d_{\fm}(A)})_s=({\lc^d_{\fm}(A)}_0)_s$ since the stable part will be concentrated in degree 0. Therefore, we have an upper bound of the $D(R,k)$-module length of $\lc^{n-d}_I(R)$ in the graded isolated singularity case: it is at most $\dim_k(\lc^d_{\fm}(A)_0)+c$. Geometrically, it is at most $\dim_k(\lc^{d-1}(X,\cO_X))+c$ where $X=\Proj(A)$.
\end{remark}

We have the following application:

\begin{example}
\label{example--Calabi-Yau hypersurface}
Let $A=k[x_1,\dots,x_n]/(f)$ where $k$ is a field of prime characteristic $p$ and $\deg(f)=n$. Denote $k[x_1,\dots,x_n]$ be $R$. Then there is a commutative diagram of short exact sequences
\[
\xymatrix{
0 \ar[r] & \lc^{n-1}_{\fm}(A)_0 \ar[r] \ar[d]^F &\lc^{n}_{\fm}(R)_{-n} \ar[r]^f \ar[d]^{f^{p-1}F} & \lc^{n}_{\fm}(R)_0=0 \ar[r] \ar[d]^F & 0\\
0 \ar[r] & \lc^{n-1}_{\fm}(A)_0 \ar[r]  &\lc^{n}_{\fm}(R)_{-n} \ar[r]^f & \lc^{n}_{\fm}(R)_0=0 \ar[r] & 0
}
\]
where $F$ denotes the natural Frobenius maps. It follows that $\dim_k(\lc^{n-1}_{\fm}(A)_0)=\dim_k(\lc^{n}_{\fm}(R)_{-n})=1$. It also follows from the diagram that $F:\lc^{n-1}_{\fm}(A)_0\to \lc^{n-1}_{\fm}(A)_0$ is injective if and only if so is
$f^{p-1}F\colon \lc^{n}_{\fm}(R)_{-n}\to \lc^{n}_{\fm}(R)_{-n}$ which holds if and only if $f^{p-1}\notin \fm^{[p]}$.

Assume further that $f$ is irreducible with an isolated singularity at $\fm$ over a perfect field $k$, {\it i.e.}, $A$ is a Calabi-Yau hypersurface over $k$. Then by Fedder's Criterion, $F\colon\lc^{n-1}_{\fm}(A)_0\to \lc^{n-1}_{\fm}(A)_0$ is injective if and only $A$ is $F$-pure. Consequently,  $\dim_k(\lc^{n-1}_{\fm}(A)_0)_s=1$ if and only $A$ is $F$-pure. Thus it follows from Theorem \ref{theorem--D-module length for isolated singularities} that:
\[l_{D(R,k)}(\lc^1_{f}(R))=\begin{cases}1& A\ {\rm is\ not\ }F{\rm -pure}\\2& {\rm otherwise}\end{cases}\]

In particular, when $\Proj(A)$ happens to be an elliptic curve, there are infinitely many primes $p$ such that $\lc^1_{f}(R)$ is a simple $D(R,k)$-module and infinitely many primes $p$ such that $\lc^1_{f}(R)$ has length 2 as a $D(R,k)$-module.

\end{example}

Next we partially generalize Theorem \ref{theorem--D-module length for isolated singularities} to the case that the dimension of the singular locus of $R/I$ is $1$. We first prove a lemma, which should be well known to experts.

\begin{lemma}
\label{lemma--localizing and completing F-module}
Let $M$ be an ($F$-finite) $F_R$-module (resp., $F^\infty_R$-module). Then $M_P$ and $M\otimes \widehat{R_P}$ are ($F$-finite) $F_{R_P}$ and $F_{\widehat{R_P}}$-modules (resp., $F^\infty_{R_P}$ and $F^\infty_{\widehat{R_P}}$-modules). Moreover, if $M$ is a simple $F_R$-module (resp., simple $F^\infty_R$-module), then $M_P$ and $M\otimes \widehat{R_P}$, if not zero, are simple as $F_{R_P}$ and $F_{\widehat{R_P}}$-modules (resp., simple as $F^\infty_{R_P}$ and $F^\infty_{\widehat{R_P}}$-modules).
\end{lemma}
\begin{proof}
The conclusion for $F$-modules follows from Proposition 2.7 and Corollary 2.9 of \cite{BlickleIntersectionhomologyDmodule}. The argument for $F^\infty_R$-modules is very similar, we omit the details.
\end{proof}
%\todo{a reference or argument for the conclusions for $D$-modules}
%Suppose $M$ is simple and $N \subseteq M_P$ is a $D_{R_P}$-submodule.
%Observe that $N \cap M$ is a $D_R$-submodule of $M$ which is not zero since we can clear any denominator.
%Therefore, $N \cap M = M$, and, because $M$ generates $M_P$ as $R_P$-module, we must have $M_P = N$.

\begin{theorem}
\label{theorem--Upper bounds for D-module length when singular locus has dimension 1}
Let $R=k[x_1,\dots,x_n]$ and $\m=(x_1,\dots,x_n)$. Let $I$ be a homogeneous reduced and equidimensional ideal of $R$. Set $A=R/I$ with $\dim A=d\geq 2$. Suppose the non-$F$-rational locus of $A$ has dimension $\leq1$ (e.g., the nonsingular locus has dimension $\leq1$). Then we have
\begin{eqnarray*}
l_{D(R,k)}(\lc_I^{n-d}(R))&\leq &c+\sum_{\dim R/P=1}\dim_{\kappa(P)} (\lc_{P\widehat{A_P}}^{d-1}(\widehat{A_P}))_s+\dim_k({\lc^d_{\fm}(A)})_s\\
&=&c+\sum_{\dim R/P=1} \dim_{\kappa(P)}(0^*_{\lc_{P\widehat{A_P}}^{d-1}(\widehat{A_P})})_s+\dim_k(0^*_{\lc^d_{\fm}(A)})_s
\end{eqnarray*}
where $c$ is number of minimal primes of $I$.
\end{theorem}
\begin{proof}
Clearly the second equality follows from Lemma \ref{lemma--stable parts of tight closure of top local cohomology modules}. Therefore it suffices to prove the first inequality. We may assume that the dimension of the non-$F$-rational locus is $1$, since otherwise the result follows from Theorem \ref{theorem--D-module length for isolated singularities} (the second term is 0). We begin with the following claim:
\begin{claim}
\label{claim--graded F-module filtration}
There exists graded $F_R$-submodules $$0\subseteq L\subseteq M\subseteq \lc_I^{n-d}(R)$$ such that every $D(R,k)$-module composition factor of $L$ is supported at a minimal prime of $A$, every $D(R,k)$-module composition factor of $M/L$ is supported at a dimension $1$ prime, and $\lc_I^{n-d}(R)/M$ is supported only at $\m$.
\end{claim}
\begin{proof}[Proof of Claim]
We have a short exact sequence
\[0\to 0^*_{\lc_{\fm}^d(A)}\to \lc_{\fm}^d(A)\to \lc_{\fm}^d(A)/0^*_{\lc_{\fm}^d(A)}\to 0\]
that induces:
\[0\to {}^*\scr{H}_{R, A}(\lc_{\fm}^d(A)/0^*_{\lc_{\fm}^d(A)})\to {}^*\scr{H}_{R, A}(\lc_{\fm}^d(A))\cong \lc_I^{n-d}(R)\to {}^*\scr{H}_{R, A}(0^*_{\lc_{\fm}^d(A)})\to 0.\]
By the graded version of \cite[Corollary 4.2 and Theorem 4.4]{BlickleIntersectionhomologyDmodule}, ${}^*\scr{H}_{R, A}(\lc_{\fm}^d(A)/0^*_{\lc_{\fm}^d(A)})$ is a direct sum of simple $D(R,k)$-modules, each supported at a different minimal prime of $A$. So we set $L={}^*\scr{H}_{R, A}(\lc_{\fm}^d(A)/0^*_{\lc_{\fm}^d(A)})$. The existence of $M$ follows by applying \cite[Theorem 2.9 (3)]{LyubeznikSinghWaltherlocalcohomologysupportedatdeterminantalideals} to ${}^*\scr{H}_{R, A}(0^*_{\lc_{\fm}^d(A)})$, which is a graded $F$-finite $F_R$-module. Note that the support of each $D(R,k)$-module composition factor of $M/L$ has dimension $1$. This is because the support of ${}^*\scr{H}_{R, A}(0^*_{\lc_{\fm}^d(A)})$ has dimension $1$ since the dimension of the non-$F$-rational locus is $1$.
\end{proof}

We know that $l_{D(R,k)}(L)=c$. Moreover, it is clear from the above claim that $M$ is the smallest $F_R$-submodule of $\lc_I^{n-d}(R)$ such that $\lc_I^{n-d}(R)/M$ is only supported at $\m$. So $\lc_I^{n-d}(R)/M$ is isomorphic, {\it as a $D(R,k)$-module}, to $E^{\oplus r}\cong\lc_\m^d(R)^{\oplus r}$ where $r=\crk\lc_I^{n-d}(R)$ by \cite[Lemma (c)]{LyubeznikInjectivedimensionofDmodulescharacteristicfreeapproach} and the definition of corank. Thus we have $l_{D(R,k)}(\lc_I^{n-d}(R)/M)=\crk\lc_I^{n-d}(R)=\dim_k(\lc_\m^d(A))_s$ by \cite[Proposition 4.10]{LyubeznikFModulesApplicationsToLocalCohomology}.

It remains to estimate the $D(R,k)$-module length of $M/L$. Suppose we have
\begin{equation}
\label{equation--D-module filtration}
0\subseteq L=M_0\subseteq M_1\subseteq M_2\subseteq\cdots\subseteq M_t=M\subseteq \lc_I^{n-d}(R)
\end{equation}
such that each $N_i=M_i/M_{i-1}$ is a simple $D(R,k)$-module. We know that each $N_i$ has a unique associated prime $P$ with $\dim R/P=1$ and $A_P$ not $F$-rational.

\begin{claim}
\label{claim--number of D-module composite factor}
The number of $N_i$ such that $\Ass(N_i)=P$ is at most $\crk \lc_{I\widehat{R_P}}^{n-d}(\widehat{R_P})$.
\end{claim}
\begin{proof}[Proof of Claim]
We localize (\ref{equation--D-module filtration}) at $P$ and complete. We have
\[
0\subseteq L\otimes\widehat{R_P}=M_0\otimes\widehat{R_P}\subseteq M_1\otimes \widehat{R_P}\subseteq\cdots \subseteq M_t\otimes \widehat{R_P}
=\lc_{I\widehat{R_P}}^{n-d}(\widehat{R_P})
\]
with successive quotients $N_i\otimes \widehat{R_P}$ (the last equality follows because $\lc_I^{n-d}(R)/M$ is supported only at $\m$). Each $N_i\otimes \widehat{R_P}$ is either $0$ or a $D(\widehat{R_P},k)$-module supported only at $P\widehat{R_P}$ (and thus a direct sum of $E(\widehat{R_P}/P\widehat{R_P})$), depending on whether $\Ass(N_i)\neq P$ or $\Ass(N_i)= P$. Therefore $\lc_{I\widehat{R_P}}^{n-d}(\widehat{R_P})/(L\otimes\widehat{R_P})$, at least as an $\widehat{R_P}$-module, is isomorphic to $E(\widehat{R_P}/P\widehat{R_P})^r$. The number of $N_i$ such that $\Ass(N_i)=P$ is thus $\leq r$.

But $L\otimes\widehat{R_P}$ is a direct sum of simple $F_{\widehat{R_P}}$-submodule of $\lc_{I\widehat{R_P}}^{n-d}(\widehat{R_P})$ supported at minimal primes of $\widehat{A_P}$ by Lemma \ref{lemma--localizing and completing F-module}, so we have $r= \crk{\lc_{I\widehat{R_P}}^{n-d}(\widehat{R_P})}$ by the definition of corank. This finishes the proof of the claim.
\end{proof}

Applying the above claim to (\ref{equation--D-module filtration}) we get:

\[l_{D(R,k)}(M/L)\leq \sum_{\dim R/P=1} \crk \lc_{I\widehat{R_P}}^{n-d}(\widehat{R_P}).\]
Because$\scr{H}_{\widehat{R_P}, \widehat{A_P} }(\lc_{P\widehat{A_P}}^{d-1}(\widehat{A_P}))
\cong \lc_{I\widehat{R_P}}^{n-d}(\widehat{R_P})$ (the indices match because $I$ is equidimensional),
by \cite[Proposition 4.10]{LyubeznikFModulesApplicationsToLocalCohomology}, we have
\[
\crk \lc_{I\widehat{R_P}}^{n-d}(\widehat{R_P})=
\dim_{\kappa(P)} (\lc_{P\widehat{A_P}}^{d-1}(\widehat{A_P}))_s.
\]
Finally, summing up the ${D(R,k)}$-module length of $L$, $M/L$, and $\lc_I^{n-d}(R)/M$, we have:
$$l_{D(R,k)}(\lc_I^{n-d}(R))\leq c+\sum_{\dim R/P=1}\dim_{\kappa(P)} (\lc_{P\widehat{A_P}}^{d-1}(\widehat{A_P}))_s+\dim_k({\lc^d_{\fm}(A)})_s.$$
\end{proof}

%\begin{remark}
%If we replace $l_D$ by $l_F$, then the analogue inequalities in Theorem \ref{theorem--D-module length in terms of dimensions of stable parts} fail, see explicit examples in Section 5. The main reason the above proof doesn't pass through is because if the residue field is not separably closed (which is always the case for $\kappa(P)$ as long as $P\neq\m$), then an $F$-finite $F_{\widehat{R_P}}$-module supported only at $P\widehat{R_P}$ is not necessarily isomorphic, {\it as an $F_{\widehat{R_P}}$-module}, to a direct sum of $E(\widehat{R_P}/P\widehat{R_P})$ \cite{HochsterfinitenesspropertyofLyubeznikFmodule}, \cite{MaThecategoryofF-moduleshasfiniteglobaldimension}.
%\end{remark}

We end this section with some remarks and questions regarding Theorem \ref{theorem--Upper bounds for D-module length when singular locus has dimension 1}:

\begin{remark}
It is clear that the sum $\sum_{\dim R/P=1}\dim_{\kappa(P)} (\lc_{P\widehat{A_P}}^{d-1}(\widehat{A_P}))_s$ in Theorem \ref{theorem--Upper bounds for D-module length when singular locus has dimension 1} is a finite sum: in fact we only need to consider those primes $P$ such that $A_P$ is not $F$-rational (which form a finite set by our assumption). We ought to point out that, more generally, without any assumption on the $F$-rational locus,  \cite[Proposition 4.14]{LyubeznikFModulesApplicationsToLocalCohomology} shows that there are only finitely many prime ideals $P$ such that $(\lc_{P\widehat{A_P}}^{j}(\widehat{A_P}))_s\neq 0$.
\end{remark}

\begin{remark}
We do not know whether the inequality in Theorem \ref{theorem--Upper bounds for D-module length when singular locus has dimension 1} is an equality. This is due to the fact that, in the proof of Claim \ref{claim--number of D-module composite factor}, we do not know whether the $D(\widehat{R_P},k)$-module $N_i\otimes \widehat{R_P}$ is isomorphic to a {\it single copy} of $E(\widehat{R_P}/P\widehat{R_P})$.
\end{remark}

In general, a simple $D(R,k)$-module may not stay simple as a $D(\widehat{R}_P,\kappa(P))$-module after taking localization and completion. We point out the following example which is derived from \cite[Example 5.1]{BlickleDmodulestructureofRFmodules}.
\begin{example}
\label{example: completion of simple no longer simple}
Let $R=k[x]$ where $k$ is an algebraically closed field of positive characteristic. Let $M=R\oplus R$ be a free $R$-module of rank $2$. We give $M$ an $F_R$-module structure by setting the composition map
\[R\oplus R\xrightarrow{\theta_M}F(R)\oplus F(R)\xrightarrow{\theta_R^{-1}\oplus\theta_R^{-1}} R\oplus R\]
to be the map represented by the matrix:
\[
\begin{pmatrix}
-x & 1 \\
1 & 0
\end{pmatrix}
\]
where $\theta_R$ denotes the standard isomorphism $R\cong F(R)$.

Next we pick a nonzero simple $F_R^\infty$-module $N\subseteq M$ (note that the only associated prime of $N$ is $0$). By \cite[Corollary 4.7]{BlickleDmodulestructureofRFmodules}, $N$ must be a simple $D(R,k)$-module since $k$ is algebraically closed. However, after we localize at $0$, that is, tensor with the fraction field $k(x)$ of $R$, $M\otimes_R k(x)$ becomes a simple $F_{k(x)}^\infty$-module because $M\otimes_R k(x)^{1/p^\infty}$ is a simple $F_{k(x)^{1/p^\infty}}^\infty$-module by \cite[Example 5.1]{BlickleDmodulestructureofRFmodules}.\footnote{Note that the matrix we used here is the inverse of the matrix as in \cite[Example 5.1]{BlickleDmodulestructureofRFmodules}. This is because we are describing the matrix representing the $F$-module structure on $M\otimes_Rk(x)^{1/p^\infty}$ while Blickle was working with the matrix representing the Frobenius action on $M\otimes_Rk(x)^{1/p^\infty}$. We leave the reader to check that they define the same $F$-module structure on $M\otimes_Rk(x)^{1/p^\infty}$.} Thus we must have $N\otimes_R k(x)\cong M\otimes_R k(x)$, but $M\otimes_R k(x)$ is not a simple $D(k(x), k(x))$-module because obviously every one-dimensional $k(x)$-subspace is a nontrivial $D(k(x), k(x))$-submodule.
\end{example}

%Hence, we ask:
%\begin{question}
%\label{question: simplicity under completion}
%Let $M$ be a simple $D(R,k)$-module and $P$ be a prime ideal of $R$. Is it true that $M\otimes_R\widehat{R}_P$ is a simple $D(\widehat{R}_P,\kappa(P))$-module?
%\end{question}

%A positive answer to Question \ref{question: simplicity under completion} will imply that we actually have an equality in Theorem \ref{theorem--Upper bounds for D-module length when singular locus has dimension 1}.

\begin{remark}
One approach to generalizing Theorem \ref{theorem--Upper bounds for D-module length when singular locus has dimension 1} is to find an $F$-submodule $M$ of $\lc^{n-d}_I(R)$ such that none of the composition factors of $M$ has 0-dimensional support and the support of $\lc^{n-d}_I(R)/M$ is contained in $\{\fm\}$. In the graded case, the existence of such an $M$ follows from the proof of \cite[Theorem 2.9]{LyubeznikSinghWaltherlocalcohomologysupportedatdeterminantalideals}. We don't know whether \cite[Theorem 2.9]{LyubeznikSinghWaltherlocalcohomologysupportedatdeterminantalideals} can be extended to non-graded case.
\end{remark}

Hence it is natural to ask the following:

\begin{question}
Does there always exist an $F$-submodule $M$ of $\lc^{n-d}_I(R)$ such that none of the composition factors of $M$ has 0-dimensional support while the support of $\lc^{n-d}_I(R)/M$ is contained in $\{\fm\}$?
\end{question}

Despite the above remarks and questions, we still expect that there should be an analogue of Theorem \ref{theorem--D-module length for isolated singularities} and Theorem \ref{theorem--Upper bounds for D-module length when singular locus has dimension 1} or similar estimates in the local case and without the restriction on the non-$F$-rational locus.

%\begin{remark}
%In the case $R=k[x_1,\dots,x_n]$, there is a general bound on the $D(R,k)$-module length of local cohomology modules: if the degree of $f$ (not necessarily %homogeneous) is $d$, then
%\[\lc^1_{f}(R)\leq (d+1)^n-1.\]
%However, this bound is not sharp at all.
%\todo{find a reference or sketch a proof of this general bound}
%\end{remark}

%\begin{remark}
%When $A=R/I$ has an isolated non-$F$-rational point, by Matlis duality we have $\dim (0^*_{\lc^d_{\fm}(A)})_s =\dim \frac{\sigma(\omega_A)}{\tau(\omega_A)}$, %where $\sigma(\omega_A)$ and $\tau(\omega_A)$ denote the parameter non-$F$-pure module and parameter test module respectively. In particular, when $f$ is an %irreducible (but not necessarily homogeneous) isolated singularity and $A=R/(f)$, we have $\dim \frac{\sigma(A)}{\tau(A)}=l_D(\lc_f^1(R))-1\leq (d+1)^n-2$ where %$d=\deg f$.
%\end{remark}

%%%%%%%%%%%%%%%%%%%%%%%%%%%%%%%%%%%%%%%%%%%%%%%%%%%%%%%%%%%%%%%%%%%%%%%%%%%%%%%%%%%%%%%%%%%%%%%%%%%%%%%%%%%%%%%%%%%%%%%%%%%%%%%%%%%%%%%%%%
\section{A lower bound on $F$-module length of local cohomology modules}
\label{section: F-pure}
%%%%%%%%%%%%%%%%%%%%%%%%%%%%%%%%%%%%%%%%%%%%%%%%%%%%%%%%%%%%%%%%%%%%%%%%%%%%%%%%%%%%%%%%%%%%%%%%%%%%%%%%%%%%%%%%%%%%%%%%%%%%%%%%%%%%%%%%%%
In this section we will give lower bounds on $l_{F_R}(\lc_I^c(R))$. Throughout this section we will still assume $R=k[[x_1,\dots, x_n]]$ or $k[x_1,\dots,x_n]$ with $\m=(x_1,\dots,x_n)$ where $k$ is a field of characteristic $p>0$, and $A=R/I$ be reduced and equidimensional or graded reduced and equidimensional of dimension $d\geq 1$.
Henceforth in this section $E=E_R(k)$ will denote the injective hull of the residue field of $R$ and
$E_A=E_A(k)= \Ann_E I$ will denote the injective hull of the residue field of $A$.

We first collect definitions and facts from \cite{SharpGradedAnnihilatorsofModulesovertheFrobeniusSkewpolynomialRingandTC} and \cite{KatzmanParameterTestIdealOfCMRings}.
Given an Artinian $A\{f\}$-module $W$, a {\it special ideal} of $W$ is an ideal of $A$ that is also the annihilator of some $A\{f\}$-submodule $V\subseteq W$, a {\it special prime} is a special ideal that is also a prime ideal (note that the special ideals depend on the $A\{f\}$-module structure on $W$, i.e., the Frobenius action $f$ on $W$).
An important result of Sharp \cite[Corollary 3.7]{SharpGradedAnnihilatorsofModulesovertheFrobeniusSkewpolynomialRingandTC} and Enescu-Hochster
\cite[Theorem 3.6]{EnescuHochsterTheFrobeniusStructureOfLocalCohomology} shows that, when $f$ acts injectively on $W$, the number of special primes of $W$ is finite.

The module of inverse polynomials $E$ comes equipped with a natural Frobenius map $T$ given by $T(\lambda x_1^{-\alpha_1} \dots x_n^{-\alpha_n})=\lambda^p x_1^{-p\alpha_1} \dots x_n^{-p\alpha_n}$
for all $\lambda\in k$ and $\alpha_1, \dots, \alpha_n\geq 0$.
Any Frobenius map on $E$ has the form $uT$ where $u\in R$ and
any Frobenius map on $E_A$ has that form with
$u\in (I^{[p]}:I)$. (cf.~\cite[Proposition 3.36]{BlickleThesis}).
Such an action $uT$ on $E_A$ is injective if and only if $u\notin \m^{[p]}$, and is nonzero if and only if $u\notin I^{[p]}$.
If we now specialize the notion of special ideals to the $A\{f\}$-module $E_A$ where $f=uT$, we see that these are
ideals $J$ such that $u\in J^{[p]}:J$
and we refer to these as  {\it $u$-special ideals}
(\cite[Theorem 4.3]{KatzmanParameterTestIdealOfCMRings}). A {\it $u$-special prime} is a $u$-special ideal that is also a prime ideal.

\subsection{$F$-pure case} Our main result in this subsection is the following:

%We first recall that, given an Artinian $A\{f\}$-module $W$, a {\it special prime} of $W$ is a prime ideal of $A$ that is also the annihilator of some $F$-stable submodule $V\subseteq W$. A result of Sharp \cite[Corollary 3.7]{SharpGradedAnnihilatorsofModulesovertheFrobeniusSkewpolynomialRingandTC} and Enescu-Hochster \cite[Theorem 3.6]{EnescuHochsterTheFrobeniusStructureOfLocalCohomology} shows that, when $F$ acts injectively on $W$, the number of special primes of $W$ is finite.

\begin{theorem}
\label{theorem--lower bounds of F-module length for F-pure}
%Let $R=k[[x_1,\dots, x_n]]$ or $k[x_1,\dots,x_n]$ with $\m=(x_1,\dots,x_n)$, and

Assume $A=R/I$ is reduced and equidimensional of dimension $d\geq 1$.
Suppose $A$ is $F$-pure. Then $l_{F_R}(\lc_I^{n-j}(R))$ is at least the number of special primes of $\lc_\m^j(A)$.
Moreover, when $A$ is quasi-Gorenstein, $l_{F_R}(\lc_I^{n-d}(R))$ is exactly the number of special primes of $\lc_\m^d(A)$.
\end{theorem}
\begin{proof}
Let $P$ be a special prime of $\lc_\m^j(A)$. Take an $A\{f\}$-submodule $N\subseteq \lc_\m^j(A)$ such that $\Ann N=P$. Recall that the Frobenius action on $N$ induces a map $F(N)\to N$. We claim that this map is surjective: let $N'\subseteq N$ be the image, we have $N'\subseteq N\subseteq \lc_\m^j(A)$ are $A\{f\}$-submodules such that the Frobenius action on $N/N'$ is nilpotent. But $\lc_\m^j(A)$ is anti-nilpotent (i.e., the Frobenius action on $H_\m^j(R)/N'$ is injective) by \cite[Theorem 3.7]{MaFinitenesspropertyoflocalcohomologyforFpurerings}. So we must have $N'=N$ and thus $F(N)\twoheadrightarrow N$ is surjective.

Taking the Matlis dual, we get $N^\vee\hookrightarrow F(N)^\vee\cong F(N^\vee)$. This shows that $N^\vee$ is a root of $\scr{H}_{R, A}(N)$ (recall that by definition, $\scr{H}_{R,A}(N)=\varinjlim(N^\vee\to F(N^\vee)\to F^2(N^\vee)\to\cdots)$). In particular, we know that the set of associated primes of $N^\vee$ is the same as the set of associated primes of $\scr{H}_{R, A}(N)$ (this follows easily from the argument in \cite[Remark 2.13]{LyubeznikFModulesApplicationsToLocalCohomology}). But $\Ann N=\Ann N^\vee=P$ and $N^\vee$ is a finitely generated $R$-module, thus $P$ is a minimal associated prime of $N^\vee$ and hence a minimal associated prime of $\scr{H}_{R,A}(N)$. This implies that $\scr{H}_{R,A}(N)$ must have a simple $F_R$-module composition factor with $P$ its unique associated prime. But we have $\lc_I^{n-j}(R)
\cong\scr{H}_{R,A}(\lc_\m^j(A))\twoheadrightarrow \scr{H}_{R,A}(N)$, hence for every special prime $P$ of $\lc_\m^j(A)$, $\lc_I^{n-j}(R)$ has a simple $F_R$-module composition factor with $P$ its unique associated prime. This proves that $l_{F_R}(\lc_I^{n-j}(R))$ is at least the number of special primes of $\lc_\m^j(A)$.

Finally, when $A$ is quasi-Gorenstein, $\lc_\m^d(A)\cong E_A$, the injective hull of the residue field. So there is a one-one correspondence between $A\{f\}$-submodules of $\lc_\m^d(A)$ and their annihilator ideals. Let $P_1,\dots,P_m$ be all the special primes with $\height P_1\geq\height P_2\geq\cdots\geq \height P_m$.  Let $Q_j=P_1\cap P_2\cap\cdots\cap P_j$. We have an ascending chain of $A\{f\}$-submodules of $\lc_\m^d(A)=E_A$: $$0\subsetneqq \Ann_EQ_1\subsetneqq\Ann_EQ_2\subsetneqq\cdots\subsetneqq\Ann_EQ_m=E_A.$$ It suffices to show that $\scr{H}_{R,A}(\Ann_EQ_j/\Ann_EQ_{j-1})$ is a nonzero simple $F$-module. It is nonzero because the Frobenius action on $\Ann_EQ_j/\Ann_EQ_{j-1}$ is not nilpotent (in fact it is injective because $\lc_\m^d(A)$ is anti-nilpotent). If it is not simple, then there exists another $A\{f\}$-submodule $M$ such that $\Ann_EQ_{j-1}\subsetneqq M\subsetneqq \Ann_EQ_{j}$, which implies $Q_j\subsetneqq \Ann M\subsetneqq Q_{j-1}$. But since the Frobenius action on $\lc_\m^d(A)=E_A$ is injective and $M$ is an $A\{f\}$-submodule, $\Ann M$ is an intersection of the special primes by \cite[Corollary 3.7]{SharpGradedAnnihilatorsofModulesovertheFrobeniusSkewpolynomialRingandTC} or \cite[Theorem 3.6]{EnescuHochsterTheFrobeniusStructureOfLocalCohomology}. This is then impossible by our height hypothesis on $P_i$ and the definition of $Q_j$.
\end{proof}

\begin{remark}
When $A$ is not $F$-pure, the number of special primes of $\lc_\m^j(A)$ is not necessarily a lower bound of $l_{F_R}(\lc_I^{n-j}(R))$. In fact, in Example \ref{example--Calabi-Yau hypersurface}, when $A=R/f$ is a Calabi-Yau hypersurface that is not $F$-pure, then $l_{F_R}(\lc_f^{1}(R))=1$ while the number of special primes of $\lc_\m^1(A)$ is 2: $(f)$ and $\m$ are both special primes of $\lc_\m^1(A)$. So the first conclusion of Theorem \ref{theorem--lower bounds of F-module length for F-pure} need not hold when $A$ is not $F$-pure.
\end{remark}

\begin{example}
Let $A=k[x_1,\dots,x_n]/(x_ix_j \mid 1\leq i<j\leq n)=R/I$. Then $A$ is a one-dimensional $F$-pure ring, and $A$ is not Gorenstein when $n\geq 3$. A straightforward computation using \cite[Theorem 5.1]{EnescuHochsterTheFrobeniusStructureOfLocalCohomology} shows that the special primes of $\lc_\m^1(A)$ are $P_i=(x_1,\dots,\widehat{x}_i,\dots,x_n)$ and $\m$ (thus there are $n+1$ special primes) but $l_{F_R}(\lc_I^{n-1}(R))=2n-1$. Hence $l_{F_R}(\lc_I^{n-1}(R))$ is strictly bigger than the number of special primes when $n\geq 3$ (and the difference can be arbitrarily large when $n\gg0$). This shows the second conclusion of Theorem \ref{theorem--lower bounds of F-module length for F-pure} need not hold when $A$ is not quasi-Gorenstein.
\end{example}

%%%%%%%%%%%%%%%%%%%%%%%%%%%%%%%%%%%%%%%%%%%%%%%%%%%%%%%%%%%%%%%%%%%%%%%%%%%%%%%%%%%%%%%%%%%%%%%%%%%%%%%%%%%%%%%%%%%%%%%%%%%%%%%%%%%%%%%%%%
\subsection{Gorenstein case: a second approach}
%%%%%%%%%%%%%%%%%%%%%%%%%%%%%%%%%%%%%%%%%%%%%%%%%%%%%%%%%%%%%%%%%%%%%%%%%%%%%%%%%%%%%%%%%%%%%%%%%%%%%%%%%%%%%%%%%%%%%%%%%%%%%%%%%%%%%%%%%%

In this subsection we assume that $A=R/I$ is Gorenstein and $F$-injective (equivalently, Gorenstein and $F$-pure). In this case $\lc^i_I(R)$ vanishes unless $i=n-d$ and we already know from Theorem \ref{theorem--lower bounds of F-module length for F-pure} that $l_{F_R}(\lc_I^{n-j}(R))$ is equal to the number of special primes of $\lc^{d}_{\mathfrak{m}}(A)$.
Our goal here is to give a more detailed analysis on the $F_R$-submodules of $\lc_I^{n-j}(R)$ in terms of their generating morphisms, and in particular we recover the second conclusion of Theorem \ref{theorem--lower bounds of F-module length for F-pure}.

{%\color{red}
Since $A$ is Gorenstein, $E_A\cong\lc^{d}_{\mathfrak{m}}(A)$ and thus there is a natural Frobenius action on $E_A$.
In this case the module $\frac{I^{[p]}:I}{I^{[p]}}$ is a cyclic $A$-module, and the natural Frobenius action on $E_A$ is given up to sign by
 $uT$, where we fix $u\in I^{[p]}:I$ whose image in $\frac{I^{[p]}:I}{I^{[p]}}$ generates it as an $A$-module (and $T$ denotes the natural Frobenius on $E$). The $u$-special ideals (resp. $u$-special primes) are thus the special ideals (resp. special primes) and they are finite by
 \cite[Corollary 3.7]{SharpGradedAnnihilatorsofModulesovertheFrobeniusSkewpolynomialRingandTC} or \cite[Theorem 3.6]{EnescuHochsterTheFrobeniusStructureOfLocalCohomology}.
}

%Since $A$ is Gorenstein, $\lc^{d}_{\mathfrak{m}}(A)$ is an injective hull $E_A(A/\fm A)$ of the residue field of $A$,
%and this is also given by $\Ann_E I$, where $E=E_R(R/\fm)$ is injective hull of the residue field of $R$.

%Matlis duality shows that any $R$-submodule of $E$ is determined by its annihilator. Consider now
%the Artinian $A\{f\}$-module $\lc^{d}_{\mathfrak{m}}(A)$ where $f$ acts as natural Frobenius action:
%the annihilators of $A\{ f \}$-submodules of $\lc^{d}_{\mathfrak{m}}(A)\cong E$ are, in the language of
%\cite{SharpGradedAnnihilatorsofModulesovertheFrobeniusSkewpolynomialRingandTC}
%\emph{special ideals of  $E$} under the given Frobenius map $f$.
%Since the action of $f$ is injective,
%\cite[Corollary 3.7]{SharpGradedAnnihilatorsofModulesovertheFrobeniusSkewpolynomialRingandTC} or \cite[Theorem 3.6]{EnescuHochsterTheFrobeniusStructureOfLocalCohomology}
%show that there are finitely many such special ideals, and that these are all the intersections of the prime special ideals.

%Any Frobenius map on $\Ann_E I$ has the form $u T$ where $T : E \rightarrow E$ is the natural Frobenius action on $E=\lc^{\dim R}_{\fm}(R)$
%and $u\in (I^{[p]}: I)$ and the special ideals of $\lc^{d}_{\mathfrak{m}}(A)$ are then those ideals $J$ of $A$ satisfying $ u J \subseteq J^{[p]}$.
%(\cite[Theorem 4.3]{KatzmanParameterTestIdealOfCMRings}). To emphasize the dependence on the Frobenius action on $E$ we will refer to these as $u$-special ideals.

Following the construction in \cite[section 4]{LyubeznikFModulesApplicationsToLocalCohomology}, we obtain a generating morphism for $\lc^{n-d}_I(R)$ of the form
$R/I \xrightarrow{u} R/I^{[p]}$ with $u$ as above.
To obtain a root, we let $K=\cup_{e\geq 1} (I^{[p^e]} : u^{1+\dots+p^{e-1}} )$ and
now $R/K \xrightarrow{u} R/K^{[p]}$ is a root of $\lc^{n-d}_I(R)$.

\begin{lemma}\label{Lemma: submodules of F-finite F-modules}
The proper $F$-finite $F$-submodules of $\lc^{n-d}_I(R)$
have roots $J/K \xrightarrow{u} J^{[p]}/K^{[p]}$
as $J$ ranges over all proper $u$-special ideals, and, furthermore, distinct special ideals $J$
define distinct $F$-finite $F$-submodules of $\lc^{n-d}_I(R)$.
\end{lemma}
\begin{proof}
\cite[Corollary 2.6]{LyubeznikFModulesApplicationsToLocalCohomology} establishes a bijection between $F$-finite $F$-submodules $\mathcal{N}$ of $\lc^{n-d}_I(R)$
and $R$-submodules of the root $R/K$ which is given by
 $\mathcal{N} \mapsto  \mathcal{N} \cap R/K$.

 Fix such $\mathcal{N}$ and write $J/K=\mathcal{N} \cap R/K$.
 The fact that $\mathcal{N}$  is a $F$-finite $F$-submodule of $\lc^{n-d}_I(R)$ implies that
 the image of the restriction of the map $R/K \xrightarrow{u} R/K^{[p]}$  to $J/K$ is in $F(J/K)=J^{[p]}/K^{[p]}$ and hence
 $J$ is a $u$-special ideal. Clearly, any such $u$-special ideal defines a $F$-finite $F$-submodule of $\lc^{n-d}_I(R)$, i.e.,
 $\displaystyle \lim_{\rightarrow} \left(J/K \xrightarrow{u} J^{[p]}/K^{[p]}\xrightarrow{u^p} J^{[p^2]}/K^{[p^2]} \xrightarrow{u^{p^2}} \dots \right)$.

 To finish the proof we need to show that any two distinct $u$-special ideals $J_1$ and $J_2$ define different
 $F$-finite $F$-submodules of $\lc^{n-d}_I(R)$. If this is not the case then for some $e\geq 0$,
 $u^{\nu_e} J_1/ K^{[p^e]} = u^{\nu_e} J_2/ K^{[p^e]}$ where $\nu_e=1+p+\dots + p^{e-1}$,
 $u^{\nu_e} J_1 + K^{[p^e]}= u^{\nu_e} J_2 + K^{[p^e]}$.

The fact that $f$ is injective on $E$ is equivalent to $u^{\nu_e}\notin L^{[p^e]}$ for all $e\geq 1$ and all
 proper ideals $L\subsetneq R$ (\cite[Theorem 4.6]{KatzmanParameterTestIdealOfCMRings}), and in particular
 $u^{\nu_e}\notin J_1^{[p^e]}$ and $u^{\nu_e}\notin J_2^{[p^e]}$.
 But now
 $u^{\nu_e} J_1 \subseteq  u^{\nu_e} J_2 + K^{[p^e]} \subseteq J_2^{[p^e]}$ and $J_2^{[p^e]}$ is a primary ideal (because $R$ is regular),
 hence $J_1 \subseteq  \sqrt{ J_2^{[p^e]} } = J_2$. Similarly, also $J_2 \subseteq   J_1$, contradicting the fact that $J_1\neq J_2$.
\end{proof}

%The discussion in the previous section now implies that all $F_R$-submodules of $\lc^{n-d}_I(R)$ have roots of the form $J/K$ where $J$ is an ideal containing $K$ with the property $u J \subseteq J^{[p]}$, i.e., $J$ is a  $u$-special ideal of $R$. These special ideals are also the special ideals of the Artinian $A\{f\}$-module $\lc^{d}_{\mathfrak{m}}(A)$ where $f$ acts as natural Frobenius action (see the beginning of subsection 5.1).

\begin{theorem}
\label{Theorem: F-length for Gorenstein rings}
Let $A=R/I$ be Gorenstein and $F$-injective where $R=k[[x_1,\dots, x_n]]$ or $k[x_1,\dots,x_n]$ with $\m=(x_1,\dots,x_n)$. Let $\{ P_1, \dots, P_m \}$ be the set of all the special prime ideals of $\lc^{d}_{\mathfrak{m}}(A)$ which contain $K$, and assume that these were ordered so that $\height P_1 \geq \height P_2 \geq \dots \geq \height P_m$. Write $Q_j=P_1 \cap \dots \cap P_j$ for all $1\leq j \leq m$.
The chain of roots
\[\xymatrix@C=5pt{
\displaystyle
0 & \subset & \frac{Q_m}{K} \ar[d]^{u} & \subset & \frac{Q_{m-1}}{K} \ar[d]^{u} & \subset & \dots & \subset & \frac{Q_1}{K} \ar[d]^{u} & \subset & \frac{R}{K} \ar[d]^{u}\\
0 & \subset & \frac{Q_m^{[p]}}{K^{[p]}} & \subset & \frac{Q_{m-1}^{[p]}}{K^{[p]}} & \subset & \dots & \subset & \frac{Q_1^{[p]}}{K^{[p]}} & \subset & \frac{R}{K^{[p]}} \\
}\]
corresponds to a maximal filtration of $\lc^{n-d}_I(R)$ in the category of $F_R$-modules.
\end{theorem}
\begin{proof}
Since $u Q_j \subseteq Q_j^{[p]}$ and $u K \subseteq K^{[p]}$ the vertical maps are well defined, and
the diagram is clearly commutative.

To show that the factors are non-zero, note that if $Q_{j+1}=P_1 \cap \dots \cap P_{j+1}=P_1 \cap \dots \cap P_j=Q_j$ then $P_{j+1} \supseteq P_1 \cap \dots \cap P_j$ and $P_{j+1} \supseteq P_i$ for some $1\leq i \leq j$.
But the ordering of $P_1, \dots, P_m$ implies that $\height P_{j+1} \leq \height P_i$, giving $P_i=P_{j+1}$, a contradiction.

If the factors are not simple, then for some $1\leq j \leq m$ there exists a special ideal $J$ such that
$P_1 \cap \dots \cap P_j \cap P_{j+1} \subsetneq J \subsetneq P_1 \cap \dots \cap P_j$.
Being special, $J$ is radical and has the form
$P_1 \cap \dots \cap P_j \cap P_{k_1} \cap P_{k_2}\cap \dots \cap P_{k_s}$ for
$j< k_1, \dots k_s \leq m$. Now for every $1\leq \ell\leq s$,
$P_{k_\ell} \supseteq P_1 \cap \dots \cap P_{j+1}$ so $P_{k_\ell} \supseteq P_w$ for some $1\leq w\leq j+1$,
and the height condition implies $P_{k_\ell} = P_w$ and we conclude that
$J\subseteq P_1 \cap \dots \cap P_{j+1}$, a contradiction.

It remains to show that the $F_R$-submodules defined by the roots $Q_j/K \xrightarrow{u} Q_j^{[p]}/K^{[p]}$ are distinct:
this follows from Lemma \ref{Lemma: submodules of F-finite F-modules}.
\end{proof}

\begin{corollary}
Suppose $A=R/I$ is Gorenstein and $F$-injective. The length of $\lc^{n-d}_I(R)$ in the category of $F_R$-modules equals the number of $u$-special primes of $\lc^{d}_{\mathfrak{m}}(A)$ that contain $K$.
\end{corollary}

\begin{remark}
Suppose $A=R/I$ is Gorenstein and $F$-injective. We can actually prove that the length of $\lc^{n-d}_I(R)$ in the category of $F^e_R$-modules for every $e$ (and hence in the category of $F_R^\infty$-modules) equals the number of $u$-special primes of $\lc^{d}_{\mathfrak{m}}(A)$ that contain $K$. It suffices to prove that every $F_R^e$-submodule of $\lc^{n-d}_I(R)$ is already an $F_R$-submodule of $\lc^{n-d}_I(R)$. By the same argument as in Lemma \ref{Lemma: submodules of F-finite F-modules}, all $F_R^e$-submodules of $\lc^{n-d}_I(R)$ have roots of the form $J/K$ where $J$ is an ideal containing $K$ with the property $u^{1+p+\cdots+p^{e-1}} J \subseteq J^{[p^e]}$. Since $A=R/I$ is Gorenstein and $F$-injective (and hence $F$-pure), $u$ is a generator of $(I^{[p]}:I)/I^{[p]}$ as an $R/I$-module and $u\notin\m^{[p]}$ by Fedder's criterion \cite{FedderFPureRat}. We will show that these imply $uJ\subseteq J^{[p]}$ and thus $J/K$ already generates an $F_R$-submodule of $\lc^{n-d}_I(R)$. Since $u\notin\m^{[p]}$ and $R$ is regular, there is a $p^{-1}$-linear map (i.e., a Frobenius splitting) $\phi$: $R\to R$ such that $\phi(u)=1$. Therefore $u^{1+p+\cdots+p^{e-1}} \in J^{[p^e]}:J$ implies $$u^{1+p+\cdots+p^{e-2}}=\phi(u^{p(1+p+\cdots+p^{e-2})}\cdot u)=\phi(u^{1+p+\cdots+p^{e-1}})\in \phi(J^{[p^e]}:J)\subseteq\phi(J^{[p^e]}:J^{[p]})=J^{[p^{e-1}]}:J$$
and thus by an easy induction we have $u\in J^{[p]}:J$ (note that we have used $(J^{[p^{e-1}]}:J)^{[p]}=J^{[p^e]}:J^{[p]}$ because $R$ is regular so the Frobenius endomorphism is flat).
\end{remark}

%%%%%%%%%%%%%%%%%%%%%%%%%%%%%%%%%%%%%%%%%%%%%%%%%%%%%%%%%%%%%%%%%%%%%%%%%%%%%%%%%%%%%%%%%%%%%%%%%%%%%%%%%%%%%%%%%%%%%%%%%%%%%%%
\subsection{Cohen-Macaulay case}
%%%%%%%%%%%%%%%%%%%%%%%%%%%%%%%%%%%%%%%%%%%%%%%%%%%%%%%%%%%%%%%%%%%%%%%%%%%%%%%%%%%%%%%%%%%%%%%%%%%%%%%%%%%%%%%%%%%%%%%%%%%%%%%
In this subsection we assume that $A=R/I$ is reduced and Cohen-Macaulay. In this case the canonical module of $A$ can be identified with an ideal $\omega\subseteq A$, let $\Omega$ be the pre-image of $\omega$ in $R$, that is, $\Omega/I=\omega \subseteq A$.
The inclusion $\omega\subseteq A$ is compatible with the Frobenius endomorphism, and the short exact sequence
$0 \rightarrow \omega \subseteq A \rightarrow A/\omega \rightarrow 0$
induces an $A$-linear map
$$0 \rightarrow \lc_\mathfrak{m}^{d-1} (A/\omega) \rightarrow  \lc_\mathfrak{m}^{d}(\omega) \rightarrow  \lc_\mathfrak{m}^{d}(A) \rightarrow  \lc_\mathfrak{m}^{d} (A/\omega) \rightarrow 0 .$$
Now each of the Artinian $A$-modules is equipped with a Frobenius map induced by the Frobenius endomorphism acting on the short exact sequence, and $\lc_\mathfrak{m}^{d} (A/\omega)$ vanishes since $\dim A/\omega < \dim A=d$. So we obtain a short exact sequence of $A\{f\}$-modules
$$0 \rightarrow \lc_\mathfrak{m}^{d-1} (A/\omega) \rightarrow  \lc_\mathfrak{m}^{d}(\omega) \rightarrow  \lc_\mathfrak{m}^{d}(A) \rightarrow  0 .$$
We can now identify $ \lc_\mathfrak{m}^{d}(\omega) $ with $E_A=\Ann_E I$.
Note that the annihilator of $\lc_\mathfrak{m}^{d-1} (A/\omega)$ in $A$ is $\omega$, hence the annihilator of $\lc_\mathfrak{m}^{d-1} (A/\omega)$ in $R$ is $\Omega$ (since $R/I=A$ and $\Omega/I=\omega$). Thus we may identify $\lc_\mathfrak{m}^{d-1} (A/\omega)$ with $\Ann_E \Omega$.

We now have a short exact sequence
$$0 \rightarrow \Ann_E \Omega \rightarrow \Ann_E I \rightarrow \lc_\mathfrak{m}^{d}(A) \rightarrow 0$$
of $A\{f\}$-modules, and we can also write  $\lc_\mathfrak{m}^{d}(A)= \Ann_E I/\Ann_E \Omega$.
Recall that any Frobenius action on $\Ann_E I$ has the form $u T$ where $T$ is the natural Frobenius on
$E$ and $u\in(I^{[p]} : I)$.
{%\color{red}
Fix $u\in R$ to be such that $u T$ is the Frobenius action on $\Ann_E I$ in the exact sequence above.
 We now can obtain an $F_R$-module filtration of  $\lc^{n-d}_I(R)=\scr{H}_{R,A}(\Ann_E I/\Ann_E \Omega)$ by applying
the Lyubeznik functor $\scr{H}_{R,A}$ to a chain of surjections
$$\Ann_E I/\Ann_E \Omega \rightarrow \Ann_E I/\Ann_E J_1 \rightarrow \dots \rightarrow  \Ann_E I/\Ann_E J_m $$
where $J_1, \dots, J_m$ are $u$-special ideals such that $\Omega \supseteq J_1 \supseteq \dots \supseteq J_m \supseteq I$. We let $K=\cup_{e\geq 1} (I^{[p^e]} : u^{1+p+\dots+p^{e-1}})$ so that $\Omega/K \xrightarrow{u} \Omega^{[p]}/K^{[p]}$ is a root for $\lc^{n-d}_I(R)=\scr{H}_{R,A}(\Ann_E I/\Ann_E \Omega)$.
}

%\todo{I didn't change the next theorem, what are the hypothesis? $S$ is Cohen-Macaulay and F-injective? or F-pure? Also there seems to be some serious typos in the last paragraph and I don't understand the argument there. Wenliang can you check with Moty to fix this?}

\begin{theorem}
\label{Theorem: F-length for CM rings}
%Let $A=R/I$ be Cohen-Macaulay where $R=k[[x_1,\dots, x_n]]$ or $k[x_1,\dots,x_n]$ with $\m=(x_1,\dots,x_n)$.
Assume $A=R/I$ is Cohen-Macaulay.
Let $\{ P_1, \dots, P_m \}$ be the set of all the $u$-special prime ideals
$P\supseteq K$
%of $\Ann_E I / \Ann_E \Omega\cong \lc^{d}_{\mathfrak{m}}(A)$
such that $P\nsupseteqq \Omega$ and $u\notin P^{[p]}$, and assume that these were ordered so that $\height P_1 \geq \height P_2 \geq \dots \geq \height P_m$.
Write $Q_j=\Omega \cap P_1 \cap \dots \cap P_j$ for all $1\leq j \leq m$.
The chain of roots
\[\xymatrix@C=5pt{
\displaystyle
0 & \subset & \frac{Q_m}{K} \ar[d]^{u} & \subset & \frac{Q_{m-1}}{K} \ar[d]^{u} & \subset & \dots & \subset & \frac{Q_1}{K} \ar[d]^{u} & \subset & \frac{\Omega}{K} \ar[d]^{u}\\
0 & \subset & \frac{Q_m^{[p]}}{K^{[p]}} & \subset & \frac{Q_{m-1}^{[p]}}{K^{[p]}} & \subset & \dots & \subset & \frac{Q_1^{[p]}}{K^{[p]}} & \subset & \frac{\Omega}{K^{[p]}} \\
}\]
corresponds to a filtration of $\lc^{n-d}_I(R)$ in the category of $F$-modules with non-zero factors.
\end{theorem}

\begin{proof}
The homomorphic images $\Ann_E I/\Ann_E Q_1, \dots, \Ann_E I/\Ann_E Q_1$ of
$\Ann_E I / \Ann_E \Omega$ are preserved by the natural Frobenius action on
$\lc^{d}_{\mathfrak{m}}(A)\cong \Ann_E I / \Ann_E \Omega$
because each $Q_i$, being the intersection of $u$-special ideals, is itself special.

An application of the Lyubeznik functor $\scr{H}_{R,A}$ to the chain of surjections
$$\Ann_E I/\Ann_E \Omega \rightarrow \Ann_E I/\Ann_E J_1 \rightarrow \dots \rightarrow  \Ann_E I/\Ann_E J_m $$
yields a filtration of
$\lc^{n-d}_I(R)$
whose generating morphisms are the vertical maps in the following commutative diagram
\[\xymatrix@C=5pt{
\displaystyle
0 & \subset & \frac{Q_m}{I} \ar[d]^{u} & \subset & \frac{Q_{m-1}}{I} \ar[d]^{u} & \subset & \dots & \subset & \frac{Q_1}{I} \ar[d]^{u} & \subset & \frac{\Omega}{I} \ar[d]^{u}\\
0 & \subset & \frac{Q_m^{[p]}}{I^{[p]}} & \subset & \frac{Q_{m-1}^{[p]}}{I^{[p]}} & \subset & \dots & \subset & \frac{Q_1^{[p]}}{I^{[p]}} & \subset & \frac{\Omega}{I^{[p]}} \\
}.
\]
We can replace these generating morphisms by their corresponding roots, and obtain the commutative diagram
\[\xymatrix@C=5pt{
\displaystyle
0 & \subset & \frac{Q_m}{K} \ar[d]^{u} & \subset & \frac{Q_{m-1}}{K} \ar[d]^{u} & \subset & \dots & \subset & \frac{Q_1}{K} \ar[d]^{u} & \subset & \frac{\Omega}{K} \ar[d]^{u}\\
0 & \subset & \frac{Q_m^{[p]}}{K^{[p]}} & \subset & \frac{Q_{m-1}^{[p]}}{K^{[p]}} & \subset & \dots & \subset & \frac{Q_1^{[p]}}{K^{[p]}} & \subset & \frac{\Omega}{K^{[p]}} \\
}\]
where now all vertical maps are roots and, once we show that the inclusions in this diagram are strict, this
gives a filtration of $\lc^{n-d}_I(R)$ with non-zero factors.

%As in the proof of Theorem \ref{Theorem: F-length for Gorenstein rings},
%if $\Omega \cap P_1 \cap \dots \cap P_i \cap P_{i+1} = \Omega \cap P_1 \cap \dots \cap P_i$ then
%$P_{i+1} \supseteq \Omega$, which is impossible.

We need to show that for all $e\geq 1$, and all $1\leq i<m$,
$$ (Q_{i+1}^{[p^e]} : u^{\nu_e}) \subsetneq (Q_{i}^{[p^e]} : u^{\nu_e}) $$
where $\nu_e=1+p+\dots+p^{e-1}$. If we have equality, we may take radicals of both sides to obtain
$$
\sqrt{(\Omega^{[p^e]} : u^{\nu_e})} \cap \bigcap_{j=1}^{i+1} \sqrt{(P_{j}^{[p^e]} : u^{\nu_e}) }
=
\sqrt{(\Omega^{[p^e]} : u^{\nu_e})} \cap \bigcap_{j=1}^i \sqrt{(P_{j}^{[p^e]} : u^{\nu_e})}
$$
and so
$$\sqrt{(P_{i+1}^{[p^e]} : u^{\nu_e})}\supseteq
\sqrt{(\Omega^{[p^e]} : u^{\nu_e})} \cap \bigcap_{j=1}^i \sqrt{(P_{j}^{[p^e]} : u^{\nu_e})} .$$
We claim that $u^{\nu_e} \notin P_{j}^{[p^e]}$ for each $j$ and we will use induction on $e$ to prove this. When $e=1$, this is precisely our assumption that $u\notin P^{[p]}_j$. Assume that $u^{\nu_e} \notin P_{j}^{[p^e]}$ and we wish to prove $u^{\nu_{e+1}} \notin P_{j}^{[p^{e+1}]}$. If $u^{\nu_{e+1}} \in P_{j}^{[p^{e+1}]}$, then we would have
\[u^{\nu_e}\in (P_{j}^{[p^{e+1}]}:u^{p^e})=(P^{[p]}_j:u)^{[p^e]}\subseteq P^{[p^e]}_j\]
a contradiction, where the last inclusion follows from the fact that $P^{[p]}_j$ is $P$-primary and our assumption that $u\notin P^{[p]}_j$.

Consequently $(P_{i+1}^{[p^e]} : u^{\nu_e})\neq R$ and we must have $(P_{i+1}^{[p^e]} : u^{\nu_e})\subseteq P_{i+1}$ and
so
$P_{i+1} \supseteq \sqrt{(\Omega^{[p^e]} : u^{\nu_e})} \cap \bigcap_{j=1}^i \sqrt{(P_{j}^{[p^e]} : u^{\nu_e})}$
and hence $P_{i+1}$ must contain one of the ideals in the intersection; since these ideals
are among the unit ideal, $P_1, \dots, P_i$ and $\sqrt{(\Omega^{[p^e]} : u^{\nu_e})}\supseteq \Omega$,
this is impossible.
\end{proof}

We have the following immediate corollary of Theorem \ref{Theorem: F-length for CM rings}.

\begin{corollary}
Let $A,R,I,u,K,\Omega$ be as in Theorem \ref{Theorem: F-length for CM rings}.
The length of $\lc^{n-d}_I(R)$ in the category of $F$-modules is at least the number of  $u$-special prime ideals
$P\supseteq K$ of $\lc^{d}_{\mathfrak{m}}(A)$ such that $P\nsupseteqq \Omega$ and $u\notin P^{[p]}$.
\end{corollary}

%\marginnote{Linquan: Could you please check that the following makes sense? MK}
\begin{remark}
If $A$ is quasi-Gorenstein and we take $\Omega=R$,
the prime special ideals in the statement of Theorem \ref{Theorem: F-length for CM rings}
are the prime special ideals $P\supseteq I$ of $\lc_\m^d(A)$ such that $P\supseteq K$, and $u\notin P^{[p]}$.
The set of all such primes has been known to be finite (\cite[Remark 5.3]{KatzmanSchwedeAlgorithmForComputing}) and
if $A$ is also $F$-injective, we obtain the same set of primes as in Theorem \ref{Theorem: F-length for Gorenstein rings};
thus Theorem \ref{Theorem: F-length for CM rings} generalizes Theorem \ref{Theorem: F-length for Gorenstein rings}.

%\color{red}
%{Unlike the Gorenstein case, we do not know how to relate the special ideals of $\lc_\m^d(A)$ under its natural Frobenius action with the $u$-special ideals (remember we pick $u$ such that the %Frobenius action $uT$ on $E_A\cong\lc_\m^d(\omega)$ induces the natural Frobenius action on $\lc_\m^d(A)$ under the surjection $\lc_\m^d(\omega)\twoheadrightarrow\lc_\m^d(A)$). In fact, even if %$A=R/I$ is Cohen-Macaulay and $F$-pure, $uT$ may not be injective: if it is injective, then $A/\omega\cong R/\Omega$ is $F$-pure since it is Gorenstein, but
%\cite[Example 3.6]{MaAsufficientconditionforFpurity} shows the latter is not always true, no matter how one choose $\omega$ and $u$.

%Since we cannot guarantee $uT$ is injective, we do not know whether the number of $u$-special primes is finite. However, the proof of Theorem \ref{Theorem: F-length for CM rings} shows that the %number of $u$-special primes that do not contain $\Omega$ and such that $u\notin P^{[p]}$, is finite.
%}
\end{remark}

\section{A computation of Fermat hypersurfaces}
\label{section: diagonal hypersurfaces}
We have seen from Theorem \ref{theorem--D-module length for isolated singularities} and Remark \ref{remark--D-module length for graded isolated singularities} that the problem of computing the $D(R,k)$-module length of $\lc_I^{n-d}(R)$ when $A=R/I$ is a graded isolated singularity comes down to computing the dimension of the Frobenius stable part of $\lc_\m^d(A)_0$. In this section we study this problem for Fermat hypersurfaces $A=k[x_0, \ldots, x_d]/(x_0^n+x_1^n+\cdots +x_d^n)$ with $d\geq 2$. We express the dimension of the Frobenius stable part of $\lc_\m^{d}(A)$ explicitly in terms of the number of solutions to a system of equations on remainders. These results generalize earlier computations of Blickle in \cite[Examples 5.26--5.29]{BlickleThesis}.

\begin{remark}
Let $A = k[x_0,x_1, \dots,x_d]/(x_0^n + x_1^n +\cdots + x_d^n)$.
Then the degree $0$ part of the top local cohomology
$\lc^d_{\fm}(A)$
has a $k$-basis consisting of
the elements of the form $\frac{x_0^c}{x_1^{a_1}\cdots x_d^{a_d}}$, where $a_1,\dots, a_d$ are positive integers and $a_1+\cdots+a_d=c\leq n-1$.
Therefore, its dimension is $\binom{n-1}{d}$.
\end{remark}

In the following we will use $s\% t$ to denote the remainder of $s$ mod $t$.

\begin{remark}
\label{remark--elementary number theory}
We want to record an elementary observation. Let $n\geq 2$ be an integer and $p$ be a prime.
Let $p = nk + r$ where $r$ is the remainder. If for some positive integer $a < n$, $n|ar$, then we claim that $p | n$.

This is because $n|ar$ and $a<n$ implies that $n$ and $r$ must have a nontrivial common divisor. But $p = nk + r$ is prime, so the only nontrivial common divisor that $n$ and $r$ could have is $p$, in which case we must have $p$ divides $n$.
\end{remark}

\begin{theorem}\label{diagonal}
\label{theorem--dimension of the stable part for diagonal hypersurface}
Let $A = k[x_0, x_1, \ldots, x_d]/(x_0^n+x_1^n + \cdots + x_d^n)$ with $d\geq 2$.
Suppose $p \equiv r \mod n$.
If $p$ does not divide $n$, then the dimension of the stable part of $\lc^{d}_{\fm}(A)$
can be computed as the number of solutions of the following system of inequalities
on $1 \leq a_i \leq n-1$
\[\begin{cases}
a_1 + \cdots + a_d < n\\
(ra_1)\% n + \cdots + (ra_d)\% n < n \\
(r^2a_1)\% n + \cdots + (r^2a_d)\% n < n \\
\qquad\vdots\\
(r^{\varphi(n)-1}a_1)\% n + \cdots + (r^{\varphi(n)-1}a_d)\% n < n
\end{cases}.\]
%If $p$ (equivalently, $r$) divides $n$, then we also need to require that $r^ha_i$ is not divisible by $n$ for all $h$ and $i$.
\end{theorem}

\begin{proof}
A basis of the degree $0$ part of $\lc_\m^d(A)$ is formed by the elements
\[\frac{x_0^c}{x_1^{a_1}\cdots x_d^{a_d}}\]
where $a_1 + \cdots + a_d = c < n$ and $a_i \geq 1$.
On such element, Frobenius acts as
\[
\frac{x_0^c}{x_1^{a_1}\cdots x_d^{a_d}} \mapsto \frac{x_0^{cp}}{x_1^{a_1p}\cdots x_d^{a_dp}}
= \frac{x_0^{(cr)\% n} \cdot (- x_1^n - \cdots - x_d^n)^{\lfloor \frac{cp}{n} \rfloor }}{x_1^{a_1p}\cdots x_d^{a_dp}}.
\]
After expanding the expression we obtain the sum of monomials of the form
\[
(-1)^{\lfloor \frac{cp}{n} \rfloor} \binom {\lfloor \frac{cp}{n} \rfloor}{\alpha_1, \ldots, \alpha_d}
\frac{x_0^{(cr)\% n} \cdot (x_1^{n\alpha_1}\cdots x_d^{n\alpha_d})}{x_1^{a_1p}\cdots x_k^{a_dp}}
\]
for $\alpha_1 + \cdots + \alpha_d = \lfloor \frac{cp}{n} \rfloor$.
This element will be zero unless $\alpha_in < a_ip$ for all $i$.
Hence it is zero if
$\alpha_i > \lfloor \frac{a_ip}{n} \rfloor$ for some $i$.
In particular, the element $\frac{x_0^c}{x_1^{a_1}\cdots x_d^{a_d}}$ is in the kernel of the Frobenius map if
\[
\alpha_1 + \cdots + \alpha_d = \left \lfloor \frac{cp}{n} \right\rfloor >
\left\lfloor \frac{a_1p}{n} \right\rfloor + \cdots + \left  \lfloor \frac{a_dp}{n} \right\rfloor, \]
i.e., $a_1p \% n + \cdots + a_dp \% n=a_1r \% n + \cdots + a_dr \% n \geq n$.
Similarly, if
\[
\left\lfloor \frac{cp}{n} \right\rfloor =
\left\lfloor \frac{a_1p}{n} \right\rfloor + \cdots + \left\lfloor \frac{a_dp}{n} \right\rfloor ,
\]
the only term that can possibly survive is
\footnotesize
\[
\binom {\lfloor \frac{cp}{n} \rfloor}{{\lfloor \frac{a_1p}{n} \rfloor}, \ldots, {\lfloor \frac{a_dp}{n} \rfloor}}
\frac {(-1)^{\lfloor \frac{cp}{n} \rfloor} x_0^{(cr)\%n} \cdot (x_1^{\lfloor \frac{a_1p}{n} \rfloor}\cdots x_d^{\lfloor \frac{a_dp}{n} \rfloor})^n}
{x_1^{a_1p}\cdots x_d^{a_dp}}
= \binom {\lfloor \frac{cp}{n} \rfloor}{{\lfloor \frac{a_1p}{n} \rfloor}, \ldots, {\lfloor \frac{a_dp}{n} \rfloor}}
\frac{(-1)^{\lfloor \frac{cp}{n} \rfloor} x_0^{(cr)\% n}}{x_1^{(a_1r)\% n}\cdots x_d^{(a_dr)\% n}}.
\]
\normalsize

Since $\lfloor \frac{cp}{n} \rfloor < p$, the binomial coefficient $ \binom {\lfloor \frac{cp}{n} \rfloor}{{\lfloor \frac{a_1p}{n} \rfloor}, \ldots, {\lfloor \frac{a_dp}{n} \rfloor}}$ is nonzero. Thus the last possibility that the above term be zero is that $a_ir$ is divisible by $n$ for some $i$. But this cannot happen as explained in Remark~\ref{remark--elementary number theory} (since this implies $p | n$).

In sum, an element of the basis
\[\frac{x_0^c}{x_1^{a_1}\cdots x_d^{a_d}}\]
is not in the kernel of Frobenius if and only if
\[(ra_1)\% n + \ldots + (ra_k)\% n < n\]
%and, if $p$ (equivalently, $r$) divides $n$, $a_ir$ are not divisible by $n$ for all $i$.
Thus the claim follows after considering further iterates of Frobenius (we only need to consider the first $\varphi(n)-1$ iterates because $r^{\varphi(n)}\equiv 1$ mod $n$ by Fermat's little theorem, so further iterates would repeat this pattern), because if a basis element is not in the kernel of all the iterates of the Frobenius map, then it contributes to an element in $\lc_\m^d(A)_s$.
\end{proof}

% \begin{corollary}
% Frobenius acts injectively on $\lc_\m^2(A)_0$ if and only if $p \equiv 1\mod n$. As a consequence, $\lc_\m^2(A)_0=(\lc_\m^2(A)_0)_s$ if and only if $p \equiv 1\mod n$.
% \end{corollary}
% \begin{proof}
% If $p \equiv 1 \mod n$ the claim immediately follows from
% Theorem~\ref{theorem--dimension of the stable part for diagonal hypersurface}.
% For the other direction, we need to show that there is a pair $1 \leq i,j \leq n-2$ such
% that $i + j < n$ but $(ri)\% n + (rj) \% n \geq n$.

% First assume that $r$ is invertible modulo $n$ and take $1 \leq i < n-1$ such that $ir \equiv -1 \mod n$.
% Then $rj \geq 1$ for all $1 \leq j$ and, thus, $(ri)\% n + (rj)\% n \geq n$ for all $j$.

% If $r$ is not invertible, $r = p$ must divide $n$. Then any pair $(i, j)$ such that $ip = n$ is not a solution.
% \end{proof}

\begin{corollary}\label{injective roots}
Suppose $p \nmid n$, then Frobenius acts injectively on $\lc_\m^{d}(A)_0$ if and only if $p \equiv 1\mod n$. As a consequence, $(\lc_\m^{d}(A)_0)_s=\lc_\m^{d}(A)_0$ if and only if $p \equiv 1\mod n$.
\end{corollary}
\begin{proof}
If $p \equiv 1 \mod n$ the claim immediately follows from
Theorem~\ref{theorem--dimension of the stable part for diagonal hypersurface}.
For the other direction, we need to show that there are integers $a_i\geq 1$ such
that $a_1 + \cdots + a_d < n$ but $(ra_1)\% n + \cdots + (ra_d) \% n \geq n$.

Now $r$ is invertible modulo $n$,  take $1 \leq a_1 < n-1$ such that $a_1r \equiv -1 \mod n$.
Then since $ra_i \geq 1$ for all $i > 1$, we always have $(ra_1)\% n + \ldots + (ra_d)\% n \geq n$ as $d\geq 2$. %If $r$ is not invertible, $p=r$ must divide $n$. Then any $a_1$ such that $pa_1 = n$ satisfies the required conditions.
\end{proof}

% \begin{corollary}
% If $p^h \equiv -1 \mod n$ for some $h$, then Frobenius acts nilpotently on $\lc_\m^2(A)_0$. As a consequence, $(\lc_\m^2(A)_0)_s=0$ in this case.
% \end{corollary}
% \begin{proof}
% Consider the equation
% \[(-i)\% n + (-j)\% n < n\]
% corresponding to $r^k$ (which is $\equiv p^h\equiv -1 \mod n$). Since $1 \leq i,j \leq n-1$, we must have $2n - i - j < n$.
% So $n < i + j$, a contradiction with the first equation of the system.
% \end{proof}

\begin{corollary}\label{nilpotent roots}
If $p^h \equiv -1 \mod n$ for some $h$, then Frobenius acts nilpotently on $\lc_\m^{d}(A)_0$. As a consequence, $(\lc_\m^d(A)_0)_s=0$ in this case.
\end{corollary}
\begin{proof}
Consider the equation
\[(-a_1)\% n + \cdots + (-a_d)\% n < n\]
corresponding to $r^h$ (which is $\equiv p^h\equiv -1 \mod n$).
Since $1 \leq a_i \leq n-1$, $(-a_1)\% n = n - a_1$, so the equation becomes
\[
dn - a_1 - \cdots - a_d < n.
\]
But this equation has no solution since $a_1 + \cdots + a_d < n$ and $d\geq 2$.
\end{proof}

\begin{remark}
The converse to the last corollary does not hold. For example, if $n = 11$ and $d = 2$, then a direct computation shows that
Frobenius acts nilpotently on $\lc^{2}_{\fm}(A)_0$ unless $p \equiv 1 \mod n$.
\end{remark}

\section{$D$-module length vs. $F$-module length}
\label{section: examples}
We continue to use the notation as in the beginning of Section 4 and Section 5. In \cite{BlickleDmodulestructureofRFmodules}, Blickle made a deep study on the comparison of $D$-module and $F$-module length. For example, in \cite[Theorem 1.1 or Corollary 4.7]{BlickleDmodulestructureofRFmodules} it was proved that if $k$ is algebraically closed, then for every $F$-finite $F^\infty_R$-module $M$, we have $l_{F^\infty_R}(M)=l_{D_R}(M)$, which is also $=l_{D(R,k)}(M)$ since $D_R=D(R,k)$ when $k$ is perfect. Moreover, when $k$ is perfect but not algebraically closed, Blickle constructed an example \cite[5.1]{BlickleDmodulestructureofRFmodules} of a simple $F_R^\infty$-module that is not $D_R$-simple (equivalently, not $D(R,k)$-simple since $k$ is perfect). In particular, even the $F_R^\infty$-module length may differ from the $D(R,k)$-module length in general.

However it is not clear whether these pathologies are artificial, i.e., can they occur for local cohomology modules with their natural $F_R$-module structure? In this section we will construct an example of a local cohomology module of $R$, with $k$ algebraically closed, such that its $F_R$-module length is strictly less than its $D_R$-module length.

To begin, let $V$ be a vector space over a field $k$ of positive characteristic $p$.
Then we can describe a $e$-th Frobenius action $f$ on $V$ in the following way.
Choose a basis $e_1, \ldots, e_n$ of $V$. Let $f(e_i) = a_{1i}e_1 + \ldots + a_{ni}e_n$.
Then for any element $b = (b_1, \ldots, b_n)^T$ written in the basis $e_i$, we can write
\[
f(b) = A b^{[p^e]}
\]
where $A = (a_{ij})$ and $[b^{p^e}]$ raises all entries to $p^e$-th power.
Or explicitly,
\[
f(b) = \begin{pmatrix}
 a_{11} & a_{12} & \cdots & a_{1n} \\
 a_{21} & a_{22} & \cdots & a_{2n} \\
 \vdots  & \vdots  & \ddots & \vdots  \\
 a_{n1} & a_{m2} & \cdots & a_{nn}
\end{pmatrix}
\begin{pmatrix}
 b_1^{p^e} \\ \vdots \\ b_n^{p^e}
\end{pmatrix}.
\]

Via this description it is very easy to see the following result.

\begin{lemma}
\label{lemma--change of basis}
Let $k$ be a field of positive characteristic $p$ and let $V$ be a finite dimensional vector space with a $e$-th Frobenius action $f$.
Let $A$ be a matrix describing the action of $f$ in some basis $e_i$.
Then in a new basis obtained by an orthogonal matrix $O$, $f$ is represented by the matrix $O A (O^{\tau})^{[p^e]}$,
where all entries of the transpose $O^{\tau}$ are raised to the $p^e$-th power.
\end{lemma}

\begin{proposition}
\label{proposition--Fmatrix}
Let $R = k[x_1, \dots, x_n]$ or $k[[x_1, \ldots, x_n]]$ and $V$ be a $k$-vector space with a $e$-th Frobenius action $f$.
Then $l_{F^e_R} (\scr{H}_{R,R}(V))$ is the length of any longest flag of $f$-subspaces
\[
0 \subset V_1 \subset V_2 \subset \ldots \subset V_m = V
\]
such that $f$ is not nilpotent on $V_i/V_{i-1}$ for every $i$.

In particular, $l_{F^e_R} (\scr{H}_{R,R}(V)) = \dim V$ if and only if there is a basis of $V$
such that, with the notation as in Lemma \ref{lemma--change of basis}, $f$ can be represented by an upper-triangular matrix $A$ with nonzero entries on the main diagonal.
\end{proposition}
\begin{proof}
The first claim is \cite[Theorem 4.7]{LyubeznikFModulesApplicationsToLocalCohomology}.

If $l_{F^e_R}(\scr{H}_{R,R}(V)) = \dim V$, then we must have $\dim V_i = i$. Then we will choose a compatible basis for the flag,
i.e., $V_i = k\langle e_1, \ldots, e_i\rangle$. Since $f(V_i) \subseteq V_i$. Now $f(e_i)= a_{i1}e_1 + \ldots + a_{ii}e_i$. Thus the matrix representing $f$ is upper-triangular. Moreover, since $f$ acts nontrivially on $V_i/V_{i-1}$, we must have $a_{ii}\neq 0$. So the matrix have nonzero entries on the main diagonal.

Conversely, if the matrix is upper-triangular with nonzero entries on the main diagonal in some basis, it is easy to see that $V_i = k\langle e_1, \ldots, e_i\rangle$ form a flag of $f$-subspaces with $f$ acts nontrivially on $V_i/V_{i-1}$.
\end{proof}

\begin{remark}\label{cycle eigenvectors}
Before proceeding further, we need a simple result in linear algebra.
Over a finite field $\mathbb{F}_p$ where $p \neq 3$, consider the matrix
\[
M=\begin{pmatrix}
0 & 0 & a \\
a & 0 & 0 \\
0 & a & 0
\end{pmatrix}.
\]
where $a\neq 0$ in $\mathbb{F}_p$. The characteristic polynomial of this matrix is
\[
\lambda^3 - a^3 = (\lambda - a)(\lambda^2 + \lambda a + a^2).
\]
The discriminant of the quadratic polynomial is $D = -3a^2$. 	
Thus if $-3$ is a quadratic residue in $\mathbb{F}_p$ then the characteristic polynomial has three distinct eigenvalues.
On the other hand, if $-3$ is not a quadratic residue, then $(1, 1, 1)$ is the only eigenvector and the restriction of the matrix $M$ on the subspace of $\mathbb{F}_p^3$ orthogonal to $(1,1,1)$ can be expressed as
\[
\begin{pmatrix}
0 & -a \\
a & -a
\end{pmatrix},
\]
which has no eigenvalues over $\mathbb{F}_p$.

By Quadratic Reciprocity, $-3$ is a quadratic residue if and only if $p = 1 \pmod 3$.
Thus if $p \neq 1 \pmod 3$, this matrix has no eigenvalues over $\mathbb{F}_{p}$ and thus cannot be transformed in an upper-triangular form by a change of basis.
Otherwise, it has three eigenvectors and can be transformed in an upper-triangular form.
\end{remark}

Our first example shows that $l_{F_R}(\lc_I^{n-c}(R))$ can be strictly less than $l_{F^\infty_R}(\lc_I^{n-c}(R))$ and thus strictly less than $l_{D(R,k)}(\lc_I^{n-c}(R))$ by (\ref{equation--basic relation on length in different categories}) if $k$ is not separably closed, even when $A=R/I$ has isolated singularities.

%\todo{Will it be better to write the following in terms of the residues modulo 21?}

\begin{corollary}
\label{corollary--D-length bigger than F-length 1}
Let $p$ be a prime number, $R = \mathbb{F}_{p}[x, y, z]$, and $f = x^7 + y^7 + z^7$.
Then
\[
l_{F_R^\infty}(\lc^1_{f} (R))=l_{D_R}(\lc^1_{f} (R))=
\begin{cases}
16 &\text { if } p \equiv 1 \pmod 7,\\
7 &\text { if } p \equiv 2 \text { or } 4 \pmod 7,\\
1 &\text { otherwise.}
\end{cases}
\]
On the other hand,
\[
l_{F_R}(\lc^1_{f} (R))=
\begin{cases}
16 &\text { if } p \equiv 1 \pmod 7,\\
7 &\text { if } p \equiv 2 \text { or } 4\pmod 7 \text{ and } p \equiv 1 \pmod 3,\\
5 &\text { if } p \equiv 2 \text { or } 4\pmod 7 \text{ and } p \not\equiv 1 \pmod 3,\\
1 &\text { otherwise.}
\end{cases}
\]
In particular,
$ l_{D_R}(\lc^1_{f} (R)) \neq l_{F_R}(\lc^1_{f} (R))$ for any $p \equiv 11 \pmod {21}$.
\end{corollary}
\begin{proof}
Let $V$ denote the $k$-vector space $(0^*_{\lc_{\fm}^2 (R/f)})_s=(\lc_{\fm}^2 (R/f)_0)_s$.
By Corollary~\ref{injective roots} and Corollary~\ref{nilpotent roots},
Frobenius acts injectively on $\lc_{\fm}^2 (R/f)_0$ if $p \equiv 1 \pmod 7$ and  nilpotently if $p \equiv 3,5,6 \pmod 7$.
When $p \equiv 2 \text { or } 4\pmod 7$, using the algorithm described in Theorem \ref{theorem--dimension of the stable part for diagonal hypersurface}
it can be checked that the Frobenius map (i.e., $e=1$) on $V$ is spanned by two $3$-cycles.
If $p \equiv 4 \pmod 7$, the cycles are
\[
\frac{z^3}{x^2y} \to \frac{z^5}{xy^4} \to \frac{z^6}{x^4y^2}
\text { and }
\frac{z^3}{xy^2} \to \frac{z^5}{x^4y} \to \frac{z^6}{x^2y^4}.
\]
While if $p \equiv 2 \pmod 7$, the cycles become
\[
\frac{z^3}{x^2y} \to \frac{z^6}{x^4y^2} \to \frac{z^5}{xy^4}
\text{ and }
\frac{z^3}{xy^2} \to \frac{z^6}{x^2y^4} \to \frac{z^5}{x^4y}.
\]
In particular, one obtains that
\[
\dim V=
\begin{cases}
15 &\text { if } p \equiv 1 \pmod 7,\\
6 &\text { if } p \equiv 2 \text { or } 4 \pmod 7,\\
0 &\text { otherwise.}
\end{cases}
\]

Since $R/(f)$ is an isolated singularity, $l_{D_R}(\lc^1_{f} (R)) = \dim V + 1$ by Theorem \ref{theorem--D-module length for isolated singularities}.
In the cases when Frobenius acts injectively or nilpotently on $\lc_{\fm}^2 (R/f)_0$, we also deduce from Proposition \ref{proposition--Fmatrix}
that $l_{F_R}(\lc^1_{f} (R)) = \dim V + 1$ (note that when Frobenius acts injectively on $\lc_{\fm}^2 (R/f)_0$, the proof of Theorem \ref{theorem--dimension of the stable part for diagonal hypersurface} shows that Frobenius sends each canonical basis element of $\lc_{\fm}^2 (R/f)_0$ to a multiple of itself, so the representing matrix is diagonal).

In the remaining cases, in order to compute its $F_R^e$-module length we will study a matrix
which represents the Frobenius map on $V$ by Proposition \ref{proposition--Fmatrix}.
We can use the proof of Theorem \ref{theorem--dimension of the stable part for diagonal hypersurface} to describe the Frobenius action on the cycles.
If $p = 7k + 4$, one obtains that the Frobenius action on both cycles are described by the matrix
\[
A = \begin{pmatrix}
0 & 0 & (-1)^{6k + 3} \binom{6k+3}{2k+1} \\
(-1)^{3k + 1} \binom{3k+1}{k} & 0 & 0 \\
0 & (-1)^{5k + 2} \binom{5k+2}{k} & 0
\end{pmatrix},
\]
where we have chosen the natural bases, e.g.
$e_1 = \frac{z^3}{x^2y}, e_2 = \frac{z^5}{xy^4}, e_3 = \frac{z^6}{x^4y^2}$ for the first cycle.
Similarly, if $p = 7k + 2$, the matrix is
\[
\begin{pmatrix}
0 & 0 & (-1)^{5k + 1} \binom{5k + 1}{4k + 1} \\
(-1)^{3k} \binom{3k}{k} & 0 & 0 \\
0 & (-1)^{6k + 1} \binom{6k+1}{4k + 1} & 0
\end{pmatrix}.
\]

We claim that the non-zero entries of $A$ are equal.
Observe that by Wilson's theorem
\[
(n-1)! (p - n)! = (n-1)! (p - n)(p - n - 1) \cdots 1 = (n-1)! (-n)(-n - 1) \cdots (-p+1) =  (-1)^{n}
\]
mod $p$. Furthermore, because $p = 7k + 4$ is odd, $k$ is odd.
Thus we can rewrite
\[
\binom{5k + 2}{k} =
\frac{(5k+2)!}{(k)!(4k+2)!}
= \frac{(p - 2k - 2)!}{(p - 6k - 4)!(4k + 2)!} = (-1)^{4k + 2} \binom{6k + 3}{2k + 1} = \binom{6k + 3}{2k + 1}
\]
and
\[
\binom{5k + 2}{k} =
\frac{(5k+2)!}{(k)!(4k+2)!}
= \frac{(p - 2k - 2)!}{k!(p - 3k - 2)!} = (-1)^{5k} \binom{3k + 1}{k} = - \binom{3k + 1}{k}.
\]
The case of $p = 7k +2$ is identical.

Since $a^{p} = a$ for any element $a \in \mathbb{F}_{p}$, the Frobenius action is linear.
Thus by Lemma \ref{lemma--change of basis} and Remark~\ref{cycle eigenvectors}, the matrix associated to the Frobenius map on the chosen basis can be transformed into upper-triangular form if and only if $p \equiv 1 \pmod 3$. Thus by Proposition \ref{proposition--Fmatrix}, in the case $p=7k+4$ or $p=7k+2$, we obtain that $l_{F_R}(\lc^1_{f} (R)) = l_{F_R}(\scr{H}_{R,R}(V)) + 1 = 7$ when $p \equiv 1 \pmod 3$, and otherwise $l_{F_R}(\lc^1_{f} (R) )= 5$.

Lastly, it is easy to see that the third iterate of the Frobenius map on $V$ can be represented by a diagonal matrix, hence
\[l_{F^3_R}(\scr{H}_{R,R}(V))=l_{F_R^3}(\scr{H}_{R,R}(0^*_{\lc_{\fm}^2 (R/f)}))=l_{F_R^\infty}(\scr{H}_{R,R}(0^*_{\lc_{\fm}^2 (R/f)}))=6\]
and $l_{F_R^\infty}(\lc^1_{f} (R))=7$.
\end{proof}

%\begin{remark}
%We suspect such pathologic behavior should be very common. As it is shown by our computations, the Frobenius action in Fermat hypersurfaces decomposes into cycles. But then it is likely that we will have no root of our polynomial, and thus no eigenvalues.
%\end{remark}

Finally, we exhibit an example of a local cohomology module of $R$, with $k$ algebraically closed, such that its $D(R,k)$-module length (equivalently, its $F_R^\infty$-module length) is strictly bigger than its $F_R$-module length. Recall that by Theorem \ref{theorem--D-module length for isolated singularities}, this cannot happen if $A=R/I$ has isolated singularities.

\begin{proposition}
\label{proposition--D-length bigger than F-length 2}
Let $p = 7k + 4$ be a prime number.
Let $R = \overline{\mathbb{F}}_{p}[x, y, z, t]$ and let $f = tx^7 + ty^7 + z^7$.
Then
\[
l_{F_R}(\lc^1_{f}(R))  = 3 < 7 = l_{F^\infty_R}(\lc^1_{f} (R))=l_{D(R,\overline{\mathbb{F}}_p)} (\lc^1_{f} (R)).
\]
\end{proposition}
\begin{proof}
Denote $R/(f)$ by $A$. First we claim that $(x,y,z)A$ is the only height-2 prime ideal of $A$ that contains the test ideal $\tau(A)$ of $A$. By \cite[Theorem 6.4]{SchwedeTuckerAsurveyoftestideals}, we have
\[\tau(A)=\sum_e\sum_{\phi\in \Hom_A(A^{1/p^e},A)}\phi(c^{1/p^e}),\]
where $c$ is a test element for $A$. According to \cite[Theorem on page 184]{HochsterFoundationsofTC}, $z^6$ is a test element for $A$ and we may set $c=z^6$.  Since $A=R/(f)$, we have
\[\tau(A)=\Big(\sum_e\sum_{\varphi\in \Hom_R(R^{1/p^e},R)}\varphi((f^{p^e-1}z^6)^{1/p^e})\Big)A.\]
It is straightforward to check that $(x^6,y^6,z^6)\subset \tau(A)$ and $\tau(A)$ has height 2. Hence $(x,y,z)A$ is the only height-2 prime ideal that contains the test ideal $\tau(A)$ of $A$. Consequently, $\widehat{A}_P$ is $F$-rational for each height-2 prime $P\neq (x,y,z)A$, or equivalently $0^*_{\lc^{2}_{P\widehat{A}_P}(\widehat{A}_P)}=0$.

Next we calculate the stable part of $\lc^3_{\fm}(A)$ where $\fm=(x,y,z,t)$. To this end, we assign the grading $\deg(x)=\deg(y)=1$, $\deg(z)=2$, and $\deg(t)=7$ degrees $1,1, 2, 7$ to $R$ and, consequently, $f$ is homogeneous. It is straightforward to check that $\lc^3_{\fm}(A)_0$ has a $\overline{\mathbb{F}}_{p}$-basis:
\[
\left [\frac{z^5}{tx^2y} \right ], \left [\frac{z^5}{txy^2}\right ], \left [\frac{z^6}{tx^4y} \right],
\left [\frac{z^6}{tx^3y^2} \right ], \left [\frac{z^6}{tx^2y^3}\right ], \left [\frac{z^6}{txy^4}\right].
\]
Because the degree of $z$ is larger that the degrees of $x$ and $y$,
each of these elements is nilpotent under the natural Frobenius action.
For example, raising the last generator to the power $p = 7k  + 4$, we get
\[
\left [\frac{z^{42k + 24}}{t^{7k + 4}x^{7k + 4}y^{28k + 16}}\right]
= \left [\frac{z^3(z^7)^{6k + 3}}{t^{7k + 4}x^{7k + 4}y^{28k + 16}}\right]
= \left [\frac{z^3(x^7 + y^7)^{6k + 3}}{t^{k + 1}x^{7k + 4}y^{28k + 16}}\right]
\]
which necessarily equals $0$ since the degree of any monomial in $x,y$ in the numerator is $7(6k+3)$ which is greater than $(7k+4)+(28k+16)$.
Hence $\lc^3_{\fm}(A)_s= (\lc^3_{\fm}(A)_0)_s=0$.

Given the grading on $R$, we are in the situation of Theorem \ref{theorem--Upper bounds for D-module length when singular locus has dimension 1}. By Claim \ref{claim--graded F-module filtration}, there exists a graded $F_R$-module filtration $0\subseteq L\subseteq M\subseteq \lc_f^1(R)$ where $L$ is supported at $(f)$, each $D(R,\overline{\mathbb{F}}_p)$-module (equivalently, $D_R$-module or $F_{R}^\infty$-module) composition factor of $M/L$ is supported at $(x,y,z)$, and $\lc_f^1(R)/M$ is supported only at $\m=(x,y,z,t)$.

By \cite[Corollary 4.2 and Theorem 4.4]{BlickleIntersectionhomologyDmodule}, $$l_{D(R,\overline{\mathbb{F}}_p)}(L)=l_{F_R^\infty}(L)=l_{F_R}(L)=1,$$ because there is only one minimal prime $(f)$ of $A$. Moreover, we have
\[
l_{D(R,\overline{\mathbb{F}}_p)}(\lc_f^1(R)/M)=l_{F_R^\infty}(\lc_f^1(R)/M)=l_{F_R}(\lc_f^1(R)/M)=0
\]
since $\dim_{\overline{\mathbb{F}}_p}(\lc_\m^3(A))_s=0$. Thus we actually have $M=\lc_f^1(R)$ in this example.

It remains to compute $l_{D(R,\overline{\mathbb{F}}_p)}(M/L)$, $l_{F_R^\infty}(M/L)$, and $l_{F_R}(M/L)$. Note that the first two are equal because we are working over an algebraically closed field \cite[Theorem 1.1]{BlickleDmodulestructureofRFmodules}. Moreover, if we take an $F_{R}^\infty$-module (resp. $F_R$-module) filtration of $M/L$, say
\[
L=M_0\subseteq M_1\subseteq M_2\subseteq\cdots\subseteq M_l=M=\lc_f^1(R),
\]
such that each $N_i=M_i/M_{i-1}$ is a simple $F_R^\infty$-module (resp. simple $F_R$-module) supported at $P=(x,y,z)$. Then if we localize at $P$ and complete, we have
\[
L\otimes\widehat{R_P}=M_0\otimes\widehat{R_P}\subseteq M_1\otimes \widehat{R_P}\subseteq\cdots \subseteq M_l\otimes \widehat{R_P}
=\lc_{f}^1(\widehat{R_P})
\]
such that each successive quotients $N_i\otimes\widehat{R_P}$ is still simple as an $F_{\widehat{R_P}}^\infty$-module (resp. simple as an $F_{\widehat{R_P}}$-module) by Lemma \ref{lemma--localizing and completing F-module}.

Observe that
\[
\widehat{R_P}/(f)\cong\overline{\mathbb{F}}_p(t)[[x,y,z]]/(tx^7+ty^7+z^7)
\]
which is an isolated singularity. Hence by Proposition \ref{proposition--Fmatrix} (and the proof of Theorem \ref{theorem--D-module length for isolated singularities}), the $F_{\widehat{R_P}}^\infty$-module (resp. $F_{\widehat{R_P}}$-module) length of $\lc_{f}^1(\widehat{R_P})/(L\otimes\widehat{R_P})$ is the longest flag of Frobenius-stable subspaces of
\[
V= \left (0^*_{\lc_{\fm}^2 (\overline{\mathbb{F}}_p(t)[[x,y,z]]/(tx^7+ty^7+z^7))} \right)_s.
\]

Via a direct computation similar to the proof of Theorem \ref{theorem--dimension of the stable part for diagonal hypersurface}
one can show that $\dim V = 6$  and $V$ is a direct sum of two three-dimensional Frobenius-stable subspaces. In the natural bases as in the proof of Corollary \ref{corollary--D-length bigger than F-length 1}, the Frobenius action on each cycle is represented by the matrix:
\[
A = \begin{pmatrix}
\!0 & \!0 & \!\binom {3k + 1}{k}t^{6k + 3} \\
\!\binom {3k + 1}{k}t^{3k + 1} & \!0 & \!0 \\
\!0 & \!\binom {3k + 1}{k}t^{5k + 2} & \!0
\end{pmatrix}
=
\binom {3k + 1}{k} \!\begin{pmatrix}
\!0 & \!0 & \!t^{6k + 3} \\
\!t^{3k + 1} & \!0 & \!0 \\
\!0 & \!t^{5k + 2} & \!0
\end{pmatrix}.
\]

We can easily see that the third iterate of the Frobenius map on $V$ can be represented by a diagonal matrix,
hence by Proposition \ref{proposition--Fmatrix}
\[l_{F_{\widehat{R_P}}^3}(\lc_{f}^1(\widehat{R_P})/(L\otimes\widehat{R_P}))=l_{F_{\widehat{R_P}}^\infty}(\lc_{f}^1(\widehat{R_P})/(L\otimes\widehat{R_P}))=6,
\]
and thus by the above discussion
\[l_{D(R,\overline{\mathbb{F}}_p)}(\lc_f^1(R))=l_{F_R^\infty}(\lc_f^1(R))=1+6=7.\]

Finally let us show that $V$ has no proper subspace stable under the Frobenius action. If $U$ is a proper Frobenius-stable subspace and $v \in U$, then we must have that $\langle v, F(v), F^2(v) \rangle \subseteq U \subsetneqq V$.
Thus if $v = (a, b, c)$ in the standard basis, then
\[
\det
\begin{pmatrix}
a & t^{6k + 3} c^p & t^{6k + 3 + (5k + 2)p } b^{p^2} \\
b & t^{3k + 1} a^p &  t^{3k + 1 + (6k + 3)p } c^{p^2} \\
c & t^{5k + 2} b^p &   t^{5k + 2 + (3k + 1)p } a^{p^2}
\end{pmatrix}
= 0
\]
Observe that if $w = \lambda v$ then
\[
\det (w, Fw, F^2w) = \lambda^{p^2 + p + 1}\det (v, Fv, F^2v),
\]
so we can multiply $v$ by the common denominator of $a, b, c$ and assume that $a, b, c \in \overline{\mathbb{F}}_p[t]$.

By a direct computation, we have
\begin{align*}
\Delta   = &\det
\begin{pmatrix}
a & t^{6k + 3} c^p & t^{6k + 3 + (5k + 2)p } b^{p^2} \\
b & t^{3k + 1} a^p &  t^{3k + 1 + (6k + 3)p } c^{p^2} \\
c & t^{5k + 2} b^p &   t^{5k + 2 + (3k + 1)p } a^{p^2}
\end{pmatrix}\\
= & t^{(3k + 1)p + 8k + 3} a^{p^2 + p + 1} - t^{(6k + 3)p + 8k + 3} ab^pc^{p^2}  -  t^{(3k + 1)p + 11k + 5} a^{p^2}bc^p \\
&  + t^{(6k + 3)p + 9k + 4} c^{p^2 + p + 1}+ t^{(5k + 2)p + 11k + 5} b^{p^2 + p + 1} - t^{(5k + 2)p + 9k + 4} a^p b^{p^2} c.
\end{align*}
Note that we may factor out $t^{(3k + 1)p + 8k + 3}$ and obtain $\Delta=t^{(3k + 1)p + 8k + 3}\Delta'$, where
\begin{multline*}
\Delta' = a^{p^2 + p + 1} - t^{(3k + 2)p} ab^pc^{p^2}  -
t^{3k + 2} a^{p^2}bc^p \\ + t^{(3k + 2)p + k + 1} c^{p^2 + p + 1}
+ t^{(2k + 1)p + 3k + 2} b^{p^2 + p + 1} - t^{(2k + 1)p + k + 1} a^p b^{p^2} c.
\end{multline*}
Let $\deg a = \alpha, \deg b = \beta, \deg c = \gamma$. Then we have
\begin{align*}
\deg a^{p^2 + p + 1} &= p^2 \alpha  + p \alpha + \alpha\\
\deg t^{(3k + 2)p} ab^pc^{p^2} &= p^2 \gamma + p (\beta + 3k + 2) + \alpha\\
\deg t^{3k + 2} a^{p^2}bc^p & = p^2 \alpha + p \gamma + \beta + 3k + 2\\
\deg t^{(3k + 2)p + k + 1} c^{p^2 + p + 1} &= p^2 \gamma + p (\gamma + 3k + 2) + \gamma + k + 1\\
\deg t^{(2k + 1)p + 3k + 2} b^{p^2 + p + 1} &= p^2 \beta + p (\beta + 2k + 1) + \beta +  3k + 2\\
\deg t^{(2k + 1)p + k + 1} a^p b^{p^2} c &= p^2 \beta + p (\alpha + 2k + 1) + \gamma + k + 1.
\end{align*}

Now we prove that there is always a nonzero term with highest degree in $\Delta'$ (hence $\Delta'$ cannot be zero). We consider the following three cases:
\begin{enumerate}
\item $\gamma\geq\max\{\alpha, \beta\}$.
\item $\beta>\gamma$ and $\beta\geq \alpha$.
\item $\alpha>\max\{\beta, \gamma\}$.
\end{enumerate}
The point is that in case (1) (similarly, in cases (2) and (3)), $\deg t^{(3k + 2)p + k + 1} c^{p^2 + p + 1}$
(resp. $\deg t^{(2k + 1)p + 3k + 2} b^{p^2 + p + 1}$, $\deg a^{p^2 + p + 1}$),
is strictly bigger than the degree of the other five terms. We give a detailed explanation in case (2) and leave the other cases to the reader to check (they are all very similar). In case (2), since $p=7k+4$, we have:
\begin{align*}
&\deg t^{(2k + 1)p + 3k + 2} b^{p^2 + p + 1}-\deg a^{p^2 + p + 1} \geq 3k+2 >0,\\
&\begin{multlined}[\textwidth]
\!\deg t^{(2k + 1)p + 3k + 2} b^{p^2 + p + 1}-\deg t^{(3k + 2)p} ab^pc^{p^2}\\
> p^2(\beta-\gamma)-p(k+1)\geq p^2-p(k+1) >0,
\end{multlined}\\
&\deg t^{(2k + 1)p + 3k + 2} b^{p^2 + p + 1}-\deg t^{3k + 2} a^{p^2}bc^p >p^2(\beta-\alpha)+p(\beta-\gamma)>0,\\
&\begin{multlined}[\textwidth]
\!\deg t^{(2k + 1)p + 3k + 2} b^{p^2 + p + 1}-\deg t^{(3k + 2)p + k + 1} c^{p^2 + p + 1} \\
>(p^2+p)(\beta-\gamma)-p(k+1)\geq p^2+p-p(k+1)>0,
\end{multlined}\\
&\deg t^{(2k + 1)p + 3k + 2} b^{p^2 + p + 1}-\deg t^{(2k + 1)p + k + 1} a^p b^{p^2} c =p(\beta-\alpha)+2k+1>0.
\end{align*}
This finishes the proof that $V$ does not have Frobenius stable subspaces and thus $$l_{F_{\widehat{R_P}}}(\lc_{f}^1(\widehat{R_P})/(L\otimes\widehat{R_P}))=2,$$ and hence by the above discussion $$l_{D(R,\overline{\mathbb{F}}_p)}(\lc_f^1(R))=l_{F_R^\infty}(\lc_f^1(R))=1+2=3.$$
\end{proof}

\begin{remark}
Blickle \cite[Section 5]{BlickleDmodulestructureofRFmodules} showed that for a general $F$-finite $F_R^\infty$-module $M$ the two lengths
$l_{F^\infty_R}(M)$ and $l_{D(R,k)} (M)$ might be different if the residue field $k$ is perfect but not algebraically closed. However, we do not know whether this can happen when $M=\lc_I^c(R)$. The Fermat hypersurfaces cannot provide such an example: the natural Frobenius
action decomposes into cycles and thus the matrix representing some large iterate of the Frobenius action will be a diagonal matrix, hence Theorem \ref{theorem--D-module length for isolated singularities} and Proposition \ref{proposition--Fmatrix} shows that the two lengths always coincide in this case.
\end{remark}

We end with a slight generalization of Blickle's result (\cite[Theorem 1.1]{BlickleDmodulestructureofRFmodules}) to local cohomology modules of rings with an isolated non-$F$-rational point over finite fields:

\begin{proposition}
\label{proposition--F-module length for isolated sing over finite field}
Let $(R,\m,k)$ and $A=R/I$ be as in the Notation at the beginning of Section 4. Assume $k$ is a finite field and $A$ has an isolated non-$F$-rational point at $\fm$. Then $l_{F^\infty_R}(\lc_I^c(R))=l_{D(R,k)}(\lc_I^c(R))$.
\end{proposition}
\begin{proof}
Following the proof of Theorem \ref{theorem--D-module length for isolated singularities} and Proposition \ref{proposition--Fmatrix}, it is enough to show that there exists $e>0$ such that the matrix $B_e$ representing the $e$-th Frobenius action on $(0^*_{\lc_\m^d(A)})_s$ is an upper triangular matrix. We will show that in fact $B_e$ is the identity matrix for $e\gg0$ sufficiently divisible. It is easy to observe that $B_e=B_1\cdot B_1^{[p]}\cdots B_1^{[p^{e-1}]}$ where $B^{[p^i]}$ denotes the matrix obtained by raising each entry of $B$ to its $p^i$-th power. Suppose $k=\mathbb{F}_{p^n}$ and thus $B^{[p^n]}=B$. Hence we know that $B_{ne}=B_1^e\cdot (B_1^{[p]})^e\cdots (B_1^{[p^{n-1}]})^e$. Now the result follows from the elementary fact that over $\mathbb{F}_{p^n}$, every invertible matrix $B$ has a power that is the identity matrix: look at the Jordan normal form of $B$ in $\overline{\mathbb{F}}_p$, taking a large $p^m$-th power will make each Jordan block into a diagonal matrix. But
an invertible diagonal matrix over $\overline{\mathbb{F}}_{p^n}$ can be raised to a large power to make each diagonal entry be $1$.
\end{proof}

\bibliographystyle{skalpha}
\bibliography{CommonBib}
\end{document}